\documentclass[11pt]{amsart}

\usepackage{amsfonts}
\usepackage{amssymb}
\usepackage{amsthm}
\usepackage{amsmath}
\usepackage{hyperref}
\usepackage{enumitem}
\usepackage{tikz}
\usepackage{pgf}
\usepackage{graphicx}
\usepackage{comment}
\usepackage{stmaryrd}
\usepackage[top=3.5cm, bottom=3.5cm, left=2.5cm, right=2.5cm]{geometry}
\usepackage[capitalise]{cleveref}

\newtheorem{prop}{Proposition}[section]
\newtheorem{theorem}[prop]{Theorem}
\newtheorem{lemma}[prop]{Lemma}
\newtheorem{corollary}[prop]{Corollary}

\theoremstyle{definition}
\newtheorem*{definition}{Definition}
\newtheorem{question}[prop]{Question}

\theoremstyle{remark}

\newtheorem*{remark*}{Remark}
\newtheorem{remark}[prop]{Remark}

\theoremstyle{plain}

\newcommand{\N}{\mathbb{N}}
\newcommand{\Z}{\mathbb{Z}}
\newcommand{\Q}{\mathbb{Q}}
\newcommand{\E}{\mathbb{E}}

\newcommand{\F}{\mathbb{F}}

\newcommand{\eps}{\varepsilon}
\newcommand{\der}{\partial}
\newcommand{\pc}{\textup{dc}}
\newcommand{\dc}{\textup{dc}}

\newcommand{\Mat}{\text{\textup{Mat}}}

\newcommand{\dd}{\textup{d}}

\newcommand{\bfe}{\mathbf{e}}

\newcommand{\bfv}{\mathbf{v}}
\newcommand{\bfw}{\mathbf{w}}

\newcommand{\Fp}{\mathbb{F}_p}
\newcommand{\mat}[2]{\Mat_{#1 \times #2}(\Fp)}
\newcommand{\codim}{\operatorname{codim}}

\numberwithin{equation}{section}

\newcommand{\e}{\varepsilon }
\def\coloneqq{\mathrel{\mathop\mathchar"303A}\mkern-1.2mu=}
\newcommand{\gen}[1]{\left\langle#1\right\rangle}

\newcommand{\dis}{\operatorname{d}}

\newcommand{\ball}{\mathbb{B}}
\newcommand{\ga}{\Gamma}

\makeatletter
\def\author@andify{%
  \nxandlist {\unskip ,\penalty-1 \space\ignorespaces}%
    {\unskip {} \@@and~}%
    {\unskip \penalty-2 \space \@@and~}%
}
\makeatother

\title{Probabilistic nilpotence in infinite groups}
\date{}
\author[A. Martino]{Armando Martino}
\address{Mathematical Sciences, University of Southampton, University Road, Southampton SO17 1BJ, United Kingdom}
\email{A.Martino@soton.ac.uk}
\author[M. C. H. Tointon]{Matthew C. H. Tointon}
\address{School of Mathematics,
University of Bristol,
Fry Building,
Woodland Road,
Bristol,
BS8 1UG,
United Kingdom}
\email{m.tointon@bristol.ac.uk}
\author[M. Valiunas]{Motiejus Valiunas}
\address{Mathematical Sciences, University of Southampton, University Road, Southampton SO17 1BJ, United Kingdom}
\curraddr{Instytut Matematyczny, Uniwersytet Wroc{\l}awski, plac Grunwaldzki 2/4, 50-384 Wroc{\l}aw, Poland}
\email{valiunas@math.uni.wroc.pl}
\author[E. Ventura]{Enric Ventura}
	
\address{Departament de Matem\`atiques, 
Universitat Polit\`ecnica de Catalunya,
 and 
Institut de Matem\`atiques de la UPC-BarcelonaTech, 
CATALONIA}
\email{Enric.Ventura@upc.edu}
\thanks{The first and third authors gratefully acknowledge the grant SEV-20150554 which partially funded their stay at the ICMAT, where a part of the research for this paper was done. The second author was supported by grant FN 200021\_163417/1 of the Swiss National Fund for scientific research. The fourth author acknowledges partial support from the Spanish Agencia Estatal de Investigaci\'on, through grant MTM2017-82740-P (AEI/ FEDER, UE), and also from the Graduate School of Mathematics through the “Mar\' {\i}a de Maeztu” Programme for Units of Excellence in R\&D (MDM-2014-0445).}

\begin{document}

\begin{abstract}The `degree of $k$-step nilpotence' of a finite group $G$ is the proportion of the tuples $(x_1,\ldots,x_{k+1})\in G^{k+1}$ for which the simple commutator $[x_1,\ldots,x_{k+1}]$ is equal to the identity. In this paper we study versions of this for an infinite group $G$, with the degree of nilpotence defined by sampling $G$ in various natural ways, such as with a random walk, or with a F\o lner sequence if $G$ is amenable. In our first main result we show that if $G$ is finitely generated then the degree of $k$-step nilpotence is positive if and only if $G$ is virtually $k$-step nilpotent (\cref{thm:dc_k}). This generalises both an earlier result of the second author treating the case $k=1$ and a result of Shalev for finite groups, and uses techniques from both of these earlier results. We also show, using the notion of polynomial mappings of groups developed by Leibman and others, that to a large extent the degree of nilpotence does not depend on the method of sampling (\cref{cor:indep}). As part of our argument we generalise a result of Leibman by showing that if $\varphi$ is a polynomial mapping into a torsion-free nilpotent group then the set of roots of $\varphi$ is sparse in a certain sense (\cref{thm:leibman}). In our second main result we consider the case where $G$ is residually finite but not necessarily finitely generated. Here we show that if the degree of $k$-step nilpotence of the finite quotients of $G$ is uniformly bounded from below then $G$ is virtually $k$-step nilpotent (\cref{thm:res.fin.}), answering a question of Shalev. As part of our proof we show that degree of nilpotence of finite groups is sub-multiplicative with respect to quotients (\cref{{thm:gallagher}}), generalising a result of Gallagher.
\end{abstract}

\dedicatory{In memory of Peter Neumann}

\maketitle

\setcounter{tocdepth}{1}
\tableofcontents

\section{Introduction}

If two elements $x,y$ are chosen independently uniformly at random from a finite group $G$, we define the probability that they commute to be the \emph{commuting probability} or \emph{degree of commutativity} of $G$, and denote it by $\pc(G)$. Peter Neumann proved the following structure theorem for groups with a high degree of commutativity.
\begin{theorem}[P. M. Neumann {\cite[Theorem 1]{neumann}}]\label{neumann}
Let $G$ be a finite group such that $\pc(G)\ge\alpha>0$. Then $G$ has a normal subgroup $\Gamma$ of index at most $\alpha^{-1}+1$ and a normal subgroup $H$ of cardinality at most $\exp(O(\alpha^{-O(1)}))$ such that $H\subset\Gamma$ and $\Gamma/H$ is abelian.
\end{theorem}

There are many natural ways in which one might seek to generalise this result. Here we seek to generalise it in two ways. The first is to higher-degree commutators. Given elements $x_i$ in a group $G$, we define the \emph{simple commutators} $[x_1,\ldots,x_k]$ inductively by setting $[x_1,x_2]=x_1^{-1}x_2^{-1}x_1x_2$ and setting $[x_1,\ldots,x_k]=[[x_1,\ldots,x_{k-1}],x_k]$. If $G$ is finite, we define $\dc^k(G)$ to be the probability that $[x_1,\ldots,x_{k+1}]=1$ if $x_1,\ldots x_{k+1}$ are chosen independently uniformly at random from $G$.

Shalev \cite{shalev} recently considered higher-order commutators in residually finite groups, proving the following results.
\begin{theorem}[Shalev; see the proof of {\cite[Theorem 1.1]{shalev}}]\label{thm:shalev}
Let $G$ be a finite group of rank at most $r$, and let $k\in\N$. Suppose that $\dc^k(G)\ge\alpha>0$. Then $G$ has a $k$-step nilpotent subgroup of index at most $O_{r,k,\alpha}(1)$.
\end{theorem}

\begin{corollary}[Shalev {\cite[Theorem 1.1]{shalev}}]\label{cor:shalev}
Let $G$ be a finitely generated residually finite group of rank at most $r$, and let $k\in\N$. Suppose that $\dc^k(G/H)\ge\alpha>0$ for every finite-index normal subgroup $H$ of $G$. Then $G$ has a $k$-step nilpotent subgroup of index at most $O_{r,k,\alpha}(1)$.
\end{corollary}

The second way in which we seek to generalise results of this type is by considering groups that are not necessarily finite or even residually finite. The first question in this setting is how to define the probability that two group elements commute. In \cite{amv} Antol\'in and the first and fourth authors approach this issue by considering sequences of finitely supported probability measures whose supports converge to the whole of $G$. Given a probability measure $\mu$ on $G$, define the \emph{degree of commutativity} $\dc_\mu(G)$ of $G$ with respect to $\mu$ via
\[
\dc_\mu(G)=\mu(\{(x,y)\in G\times G:xy=yx\})
\]
(here, and throughout, we abuse notation slightly by writing $\mu(X)$ for $(\mu\times\cdots\times\mu)(X)$ when $X\subset G^k$).
Then, given a sequence $M=(\mu_n)_{n=1}^\infty$ of probability measures on $G$, define the \emph{degree of commutativity} $\dc_M(G)$ of $G$ with respect to $M$ via
\[
\dc_M(G)=\limsup_{n\to\infty}\dc_{\mu_n}(G).
\]

Here we extend this notion to more general equations. For each $k\in\N$, write $F_k$ for the free group on $k$ generators, denoted $x_1,\ldots,x_k$.

\begin{definition}Let $G$ be a group.
\begin{enumerate}
\item An \emph{equation} in $k$ variables over $G$ is a word $\varphi\in F_k\ast G$. Abusing notation slightly, we may view $\varphi$ as a function $G^k\to G$ by defining $\varphi(g_1,\ldots,g_k)$ to be the element of $G$ resulting from replacing each instance of $x_i$ in the word $\varphi$ by $g_i$.
\item Given a probability measure $\mu$ on $G$ and an equation $\varphi$ in $k$ variables over $G$, define the \emph{degree of satisfiability} $\dd\varphi_\mu(G)$ of $\varphi$ in $G$ with respect to $\mu$ via
\[
\dd\varphi_\mu(G) = \mu(\{ (g_1,\ldots,g_k) \in G^k \mid \varphi(g_1,\ldots,g_k) = 1 \}).
\]
Then, given a sequence $M=(\mu_n)_{n=1}^\infty$ of probability measures on $G$, define the \emph{degree of satisfiability} $\dd\varphi_M(G)$ of $\varphi$ in $G$ with respect to $M$ via
\[
\dd\varphi_M(G)=\limsup_{n\to\infty}\dd\varphi_{\mu_n}(G).
\]
When $G$ is finite and $\mu$ is the uniform probability measure on $G$ we write simply $\dd\varphi(G)=\dd\varphi_\mu(G)$.
\end{enumerate}
\end{definition}

In particular, if $c = [x_1,x_2] \in F_2$ is a commutator, then we obtain the usual definitions of $\dc_\mu(G)$ and $\dc_M(G)$, as above. More generally, here and throughout we denote by $c^{(k)}$ the $(k+1)$-fold simple commutator, $c^{(k)} = [x_1,\ldots,x_{k+1}] \in F_{k+1}$, so that $c = c^{(1)}$. We call the resulting number $\dc_\mu^k(G)$ (respectively $\dc_M^k(G)$) the \emph{degree of $k$-nilpotence} of $G$ with respect to $\mu$ (respectively $M$).
For notational convenience in the inductive proof of \cref{thm:dc_k}, below, we also define $c^{(0)} = x_1 \in F_1$, so that $\dc_M^0(G)=\limsup_{n\to\infty}\mu_n(\{1\})$.

In \cite{amv}  Antol\'in and the first and fourth authors suggest that for any `reasonable' sequence $M=(\mu_n)_{n=1}^\infty$ of probability measures on $G$ we should have $\dc_M^k(G)>0$ if and only if $G$ is virtually $k$-step nilpotent. They further suggest that `reasonable' might mean that the measures $\mu_n$ cover $G$ with `enough homogeneity' as $n\to\infty$. A specific example they give of what should be a `reasonable' sequence is where $\mu$ is some finite probability measure on $G$, and $\mu_n=\mu^{\ast n}$ is defined by letting $\mu^{\ast n}(x)$ be the probability that a random walk of length $n$ on $G$ with respect to $\mu$ ends at $x$. If $G$ is amenable, another natural sequence of measures to consider is the sequence of uniform probability measures on a F\o lner sequence, or more generally an \emph{almost-invariant sequence of measures}, which is to say a sequence $(\mu_n)_{n=1}^\infty$ of probability measures satisfying
\[
\|x\cdot\mu_n-\mu_n\|_1\to0
\]
for every $x\in G$ (here $x\cdot\mu$ is defined by setting $x\cdot\mu(A)=\mu(x^{-1}A)$).

In \cite{comm.prob} the second author gave some fairly general conditions on a sequence $(\mu_n)_{n=1}^\infty$ of measures under which such a theorem holds in the case $k=1$. The following specific cases follow from \cite[Theorems 1.9, 1.11 \& 1.12]{comm.prob}.
\begin{theorem}[\cite{comm.prob}]\label{thm:dc.specific}
Let $G$ be a finitely generated group. Suppose that either
\begin{enumerate}
\item $\mu$ is a symmetric, finitely supported generating probability measure on $G$ with $\mu(\{1\})>0$, and $M=(\mu^{\ast n})_{n=1}^\infty$ is the sequence of measures corresponding to the steps of the random walk on $G$ with respect to $\mu$; or
\item $G$ is amenable and $M=(\mu_n)_{n=1}^\infty$ is an almost-invariant sequence of probability measures on $G$.
\end{enumerate}
Suppose that $\dc_M(G)\ge\alpha>0$. Then $G$ has a normal subgroup $\Gamma$ of index at most $\lceil\alpha^{-1}\rceil$ and a normal subgroup $H$ of cardinality at most $\exp(O(\alpha^{-O(1)}))$ such that $H\subset\Gamma$ and $\Gamma/H$ is abelian. In particular, if the rank of $G$ is at most $r$ then $G$ has an abelian subgroup of index at most $O_{r,\alpha}(1)$.
\end{theorem}

One of the main aims of \cite{comm.prob} was to provide a concrete but more general set of hypotheses on $M$ under which \cref{thm:dc.specific} holds. This led to the following definitions.

\begin{definition}[uniform detection of index]
Let $\pi:(0,1]\to(0,1]$ be a non-decreasing function such that $\pi(\gamma)\to0$ as $\gamma\to0$. We say that a sequence $M=(\mu_n)_{n=1}^\infty$ of probability measures on a group $G$ \emph{detects index uniformly at rate $\pi$} if for every $\eps>0$ there exists $N=N(\eps)\in\N$ such that for every $m\in\N$ if $[G:H]\ge m$ then $\mu_n(H)\le\pi(\frac{1}{m})+\eps$ for every $n\ge N$. We also say simply that $M$ \emph{detects index uniformly} to mean that there exists some $\pi$ such that $M$ detects index uniformly at rate $\pi$.
\end{definition}
The word `uniform' in the above definition refers to the requirement that the definition be satisfied by the same $N(\eps)$ for all subgroups $H$.

\begin{definition}[uniform measurement of index]
We say that a sequence $M=(\mu_n)_{n=1}^\infty$ of probability measures on a group $G$ \emph{measures index uniformly} if $\mu_n(xH)\to1/[G:H]$ uniformly over all $x\in G$ and all subgroups $H$ of $G$ (here we define $1/[G:H]=0$ if $[G:H]=\infty$).
\end{definition}
Note that if a sequence of probability measures on a group measures index uniformly then it also detects index uniformly with rate $\iota:(0,1]\to(0,1]$ defined by $\iota(x)=x$.

The second author shows in \cite[Theorems 1.11 \& 1.12]{comm.prob} that on a finitely generated group every sequence of measures corresponding to the steps of a random walk measures index uniformly, as does every almost-invariant sequence of measures. This is a key ingredient in the proof of \cref{thm:dc.specific}.

In the present paper we combine Shalev's techniques with those of \cite{comm.prob} to generalise \cref{thm:shalev} similarly to arbitrary finitely generated groups, as follows.
\begin{theorem}\label{thm:dc_k}
Let $G$ be a finitely generated group of rank at most $r$, and let $M=(\mu_n)_{n=1}^\infty$ be a sequence of measures that
detects index uniformly at rate $\pi$. 
Suppose that $\dc_M^k(G)\ge\alpha>0$. Then $G$ has a $k$-step nilpotent subgroup of index at most
$O_{r,k,\pi,\alpha}(1)$.
\end{theorem}
The following specific cases of interest of \cref{thm:dc_k} then follow from \cite[Theorems 1.11 \& 1.12]{comm.prob}.
\begin{theorem}\label{thm:dc_k.specific}
Let $G$ be a finitely generated group of rank at most $r$, and let $k\in\N$. Suppose that either
\begin{enumerate}
\item $\mu$ is a symmetric, finitely supported generating probability measure on $G$ with $\mu(\{1\})>0$, and $M=(\mu^{\ast n})_{n=1}^\infty$ is the sequence of measures corresponding to the steps of the random walk on $G$ with respect to $\mu$; or
\item $G$ is amenable and $M=(\mu_n)_{n=1}^\infty$ is an almost-invariant sequence of probability measures on $G$.
\end{enumerate}
Suppose that $\dc_M^k(G)\ge\alpha>0$. Then $G$ has a $k$-step nilpotent subgroup of index at most $O_{r,k,\alpha}(1)$.
\end{theorem}
We prove \cref{thm:dc_k} in \cref{sec:quant}.

\bigskip
Shalev actually proves a slightly more general result than \cref{thm:shalev}. Given a finite group $G$ and an element $g\in G$, write
\[
P^k(G,g)=\mu(\{(x_1,\ldots,x_{k+1})\in G^{k+1}:[x_1,\ldots,x_{k+1}]=g\}),
\]
noting that $\dc^k(G)=P^k(G,1)$. What Shalev shows is that \cref{thm:shalev} remains true if the assumption that $\dc^k(G)\ge\alpha>0$ is replaced by the weaker assumption that $P^k(G,g)\ge\alpha>0$ for some $g\in G$.

We can adapt the statement of \cref{thm:dc_k} similarly. First, given a probability measure $\mu$ on a group $G$, define
\[
P_\mu^k(G,g)=\mu(\{(x_1,\ldots,x_{k+1})\in G^{k+1}:[x_1,\ldots,x_{k+1}]=g\}).
\]
Then, given a sequence $M=(\mu_n)_{n=1}^\infty$ of probability measures on $G$, define
\[
P_M^k(G,g)=\limsup_{n\to\infty}P_{\mu_n}^k(G,g).
\]
\begin{prop}\label{prop:non-id}
Let $G$ be a group, and let $M=(\mu_n)_{n=1}^\infty$ be a sequence of probability measures on $G$ that measures index uniformly. Then $P_M^k(G,1)\ge P_M^k(G,g)$ for every $g\in G$.
\end{prop}
Combined with \cref{thm:dc_k} this immediately gives the following.
\begin{corollary}
Let $G$ be a finitely generated group of rank at most $r$, and let $M=(\mu_n)_{n=1}^\infty$ be a sequence of measures that measures index uniformly on $G$. Let $g\in G$, and suppose that $P_M^k(G,g)\ge\alpha>0$. Then $G$ has a $k$-step nilpotent subgroup of index at most $O_{r,k,\alpha}(1)$.
\end{corollary}
We prove \cref{prop:non-id} in \cref{sec:quant}.

\bigskip
It is easy to see that if a finitely generated group $G$ has a nilpotent subgroup of finite index then $\dc_M^k(G)>0$ for every sequence $M$ of measures measuring index uniformly on $G$. The conclusion of \cref{thm:dc_k} is therefore qualitatively optimal. Note, however, that \cref{thm:dc.specific} shows that in the case $k=1$ \cref{thm:dc_k} can be improved quantitatively---in the sense that the bounds can be made independent of the rank of $G$---at the expense of concluding that $G$ is bounded-by-abelian-by-bounded as in \cref{neumann}, rather than virtually abelian.

%
%

The following result suggests that a quantitatively optimal result for $\dc_M^k(G)$ must also allow for bounded-by-($k$-step nilpotent)-by-bounded groups in its conclusion.
\begin{prop}\label{prop:converse}
Let $m,d,k\in\N$. Let $G$ be a finitely generated group, let $\Gamma$ be a subgroup of $G$ of index at most $m$, and let $H$ be a subgroup of cardinality at most $d$ such that $\Gamma/H$ is $k$-step nilpotent. Let $M=(\mu_n)_{n=1}^\infty$ be a sequence of measures that measures index uniformly on $G$. Then $\dc_M^k(G)\ge\frac{1}{m^{k+1}d}$.
\end{prop}
However, Eberhard \cite[\S2.7]{eberhard.thesis} has shown that finite groups $G$ satisfying $\dc_M^k(G)\ge1/4$ are not necessarily bounded-by-($2$-step nilpotent)-by-bounded.
\begin{question}
Is there a `nice' quantitative algebraic characterisation of the finitely generated groups $G$ satisfying $\dc_M^k(G)\ge\alpha>0$ with bounds that do not depend on the rank of $G$?
\end{question}
This issue has also been raised by Ben Green (private communication). In \cref{sec:rank} we give examples for all $k \geq 1$ to show that the dependence of the bound on the rank is necessary in \cref{thm:dc_k} as stated.

\cref{thm:dc_k,prop:converse} combine to give a new proof of the following folklore result.

\begin{corollary}[finitely generated finite-by-($k$-step nilpotent) groups are virtually $k$-step nilpotent]\label{cor:fbn}
Let $r,d,k\in\N$. Let $G$ be a finitely generated group of rank at most $r$, and let $H$ be a subgroup of cardinality at most $d$ such that $G/H$ is $k$-step nilpotent. Then $G$ contains a $k$-step nilpotent subgroup of index at most $O_{r,d,k}(1)$.
\end{corollary}

In \cref{prop:needfg} we show that a finite-by-($k$-step nilpotent) group $G$ need not be virtually $k$-step nilpotent, and so the assumption that $G$ is finitely generated is necessary in \cref{cor:fbn}. However, if such a group $G$ is in addition assumed to be residually finite then we can deduce that $G$ is virtually $k$-step nilpotent, regardless of finite generation: see \cref{rem:rfnilp}.

\bigskip\noindent\textsc{Equations over virtually nilpotent groups.}
The second author shows in \cite[Theorem 1.17]{comm.prob} that if $G$ is a finitely generated group and  $M=(\mu_n)_{n=1}^\infty$ is a sequence of measures that measures index uniformly on $G$ then the $\limsup$ in the definition of $\dc_M$ is actually a limit, and that this limit does not depend on the choice of $M$. In the present work we extend this to $\dc_M^k$ for $k\ge2$, as follows.

\begin{theorem}\label{cor:indep}
Let $G$ be a finitely generated group. Then $\dc_M^k(G)$ takes the same value for all sequences $M$ of measures that measure index uniformly on $G$, and for every such sequence $M$ the $\limsup$ in the definition of $\dc_M^k(G)$ actually a limit.
\end{theorem}

In view of \cref{thm:dc_k}, in proving \cref{cor:indep} it is enough to consider virtually nilpotent groups, and in that context we actually prove something more general:~we show that for \emph{any} equation $\varphi$ over a finitely generated virtually nilpotent group $G$ the numbers $\dd\varphi_M(G)$ are well behaved in the sense of \cref{cor:indep}. To do this we use a notion of sparsity that is independent of any particular sequence of measures, as follows.
\begin{definition}
Given a group $G$, a set $V\subset G$ is said to be \emph{negligible by finite quotients of $G$} if for every $\eps>0$ there exists a finite-index normal subgroup $N\lhd G$ such that $|VN/N|\le\eps|G/N|$.
\end{definition}
The utility of this definition lies in the fact that if $V\subset G^k$ is negligible by finite quotients of $G^k$ then $\mu_n(V)\to0$ as $n\to\infty$ for every sequence $(\mu_n)_{n=1}^\infty$ of measures that measure index uniformly on $G$; we prove this below as \cref{prop:fin.quot.=>seq}.

It is worth noting that, in the language of profinite groups, a subset $V \subset G$ is negligible by finite quotients if and only if its closure in the profinite completion $\widehat{G}$ of $G$ has zero Haar measure. Indeed, this happens if and only if for any $\eps > 0$ we can find an open subset $U \subseteq \widehat{G}$ of Haar measure $\geq 1-\eps$ such that $U \cap V = \varnothing$.

Our result is then as follows.
\begin{theorem}\label{thm:v.nilp.indep}
Let $G$ be a finitely generated virtually nilpotent group, and let $N$ be a torsion-free nilpotent normal subgroup of finite index in $G$. Let $\varphi$ be an equation in $k$ variables over $G$. Then the set
\[
G_\varphi=\{(g_1,\ldots,g_k)\in G^k:\varphi(g_1,\ldots,g_k)=1\}
\]
of solutions to $\varphi$ is the union of a set of cosets of $N^k$ and a set that is negligible by finite quotients of $G^k$.
\end{theorem}
Recall that virtually nilpotent groups are always virtually torsion-free, so by defining the subgroup $N$ in \cref{thm:v.nilp.indep} we are merely fixing notation, rather than imposing an additional hypothesis. In particular, \cref{thm:v.nilp.indep,prop:fin.quot.=>seq} immediately imply the following result.
\begin{corollary}\label{cor:indep.arbitrary}
Let $G$ be a finitely generated virtually nilpotent group, and let $\varphi$ be an equation over $G$. Then $\dd\varphi_M(G)$ is the same for all sequences $M$ of measures that measure index uniformly on $G$, and for every such sequence $M$ the $\limsup$ in the definition of $\dd\varphi_M(G)$ is actually a limit.
\end{corollary}
In particular, combined with \cref{thm:dc_k} this implies \cref{cor:indep}. Indeed, in the case of a residually finite group \cref{thm:dc_k,thm:v.nilp.indep} even give the value of $\dc_M^k(G)$ in terms of $\dc^k$ of the finite quotients of $G$, as follows.
\begin{corollary}\label{cor:indep.resid.fin}
Let $G$ be a residually finite group and let $M$ be a sequence of measures on $G$ that measures index uniformly. Let $H_1>H_2>\cdots$ be a sequence of finite-index normal subgroups of $G$ such that $\bigcap_{m=1}^\infty H_m=\{1\}$. Then $\dc^k(G/H_m)\to\dc_M^k(G)$ as $m\to\infty$.
\end{corollary}
Note also that if $G$ is assumed a priori to be virtually nilpotent then \cref{thm:v.nilp.indep} similarly gives the value of $\dd\varphi_M(G)$ in \cref{cor:indep.arbitrary}.
\begin{corollary}\label{cor:resid.fin.arb}
Let $G$ be a finitely generated virtually nilpotent group, and let $\varphi$ be an equation over $G$. Let $H_1>H_2>\cdots$ be a sequence of finite-index normal subgroups of $G$ such that $\bigcap_{m=1}^\infty H_m=\{1\}$. Then $\dd\varphi(G/H_m)\to\dd\varphi_M(G)$ as $m\to\infty$.
\end{corollary}
\begin{remark}
It is easy to see, for $G$, $M$ and $(H_i)$ as in \cref{cor:indep.resid.fin} and $\varphi$ an arbitrary equation over $G$, that the sequence $\left( \dd\varphi(G/H_m) \right)_{m=1}^\infty$ is decreasing and bounded below by $0$---and hence converges to some limit---and that
\begin{equation}\label{eq:fin.quot.lim}
\dd\varphi_M(G) \leq \lim_{m \to \infty} \dd\varphi(G/H_m).
\end{equation}
\cref{cor:indep.resid.fin,cor:resid.fin.arb} say that if $\varphi = c^{(k)}$, or if $G$ is virtually nilpotent, then we have equality in \eqref{eq:fin.quot.lim}.
\end{remark}

An equation $\varphi$ in $k$ variables over a group $G$ is a \emph{probabilistic identity} with respect to a sequence $M$ of measures if $\dd\varphi_M(G) > 0$; it is a \emph{coset identity} if there exists a finite-index subgroup $H<G$ and elements $g_1,\ldots,g_k\in G$ such that $\varphi(g_1H,\ldots,g_kH)=1$. Shalev \cite[Corollary 1.2]{shalev} notes that \cref{thm:shalev} implies that if $[x_1,\ldots,x_{k+1}]$ is a coset identity in a finitely generated residually finite group $G$ then $G$ is virtually $k$-step nilpotent. The following is immediate from \cref{thm:v.nilp.indep}.
\begin{corollary}
Let $G$ be a finitely generated virtually nilpotent group, and let $M$ be a sequence of measures on $G$ that measures index uniformly. Then an equation $\varphi$ in $G$ is a probabilistic identity with respect to $M$ if and only if it is a coset identity.
\end{corollary}

We prove \cref{thm:v.nilp.indep} in \cref{sec:polys,sec:v.nilp} using the notion of \emph{polynomial mappings} of groups developed by Leibman and others.

\bigskip\noindent\textsc{Finite quotients.}
In the last of our main results we remove the `finitely generated' assumption from \cref{cor:shalev}, answering a question posed by Shalev \cite[Problem 3.1]{shalev}.
\begin{theorem} \label{thm:res.fin.}
Let $G$ be a residually finite group, and let $\mathcal N$ be a family of finite-index normal subgroups of $G$ that is closed under finite intersections and such that $\bigcap_{N\in\mathcal N}N=\{1\}$. Let $k\in\N$, and suppose that there exists a constant $\alpha>0$ such that $\dc^k(G/N)\ge\alpha$ for every $N\in\mathcal N$. Then $G$ has a $k$-step nilpotent subgroup of finite index.
\end{theorem}

This can be restated as follows: if $G$ is a profinite group and the set $\mathcal{N}_k(G) = \{ (x_1,\ldots,x_{k+1}) \in G^{k+1} \mid [x_1,\ldots,x_{k+1}]=1 \}$ has positive Haar measure in $G^{k+1}$, then $G$ has an open $k$-step nilpotent subgroup. Indeed, the statement for a profinite group $G$ can be deduced by taking $\mathcal{N}$ in \cref{thm:res.fin.} to be a collection of normal subgroups forming a neighbourhood basis for the identity in $G$, and \cref{thm:res.fin.} can be deduced from this statement by considering the profinite closure $\widehat{G}$ of a given residually finite group $G$.

Thus, \cref{thm:res.fin.} generalises a result of L\'evai and Pyber \cite[Theorem 1.1 (iii)]{lp}, who prove it in the case $k=1$. This was generalised in a different direction by Hofmann and Russo, who prove in \cite[Theorem 1.2]{hr} that any compact group $G$ such that $\mathcal{N}_1(G)$ has positive Haar measure in $G^2$ has an open abelian subgroup. We therefore ask whether the following is true.
\begin{question}\label{qu:compact}
Let $G$ be a compact group, and suppose that $\mathcal{N}_k(G) = \{ (x_1,\ldots,x_{k+1}) \in G^{k+1} \mid [x_1,\ldots,x_{k+1}]=1 \}$ has positive Haar measure in $G^{k+1}$. Does $G$ have an open $k$-step nilpotent subgroup?
\end{question}

In \cite[\S1]{lp}, L\'evai and Pyber also note that, for $k = 1$, the finite index of the abelian subgroup in \cref{thm:res.fin.} need not be bounded in terms of $\alpha$, citing the examples of direct products of abelian groups and extra-special groups. In \cref{sec:rank} we generalise these examples to show that, for any $k \geq 1$, the index of the $k$-step nilpotent subgroup coming from \cref{thm:res.fin.} need not be bounded in terms of $k$ and $\alpha$.

The key ingredient in the proof of \cref{thm:res.fin.} is the following result of independent interest on finite groups.
\begin{theorem} \label{thm:gallagher}
Let $G$ be a finite group and let $N \trianglelefteq G$. Then $\dc^k(G) \leq \dc^k(N) \dc^k(G/N)$ for all $k\in\N$.
\end{theorem}
\cref{thm:gallagher} generalises the main result of Gallagher \cite{gallagher}, who proves it for $k = 1$. It also generalises a theorem by Moghaddam, Salemkar and Chiti \cite[Theorem A]{msc}, who prove \cref{thm:gallagher} when the centraliser of every element of $G$ is normal. As is noted in \cite[Section 2]{erl}, the centraliser of every element of $G$ being normal implies that $G$ is $3$-step nilpotent; this is an extremely strong hypothesis for \cref{thm:gallagher}, rendering it trivial for $k\ge3$, for example.

We prove \cref{thm:res.fin.,thm:gallagher} in \cref{sec:gallagher}.

\bigskip\noindent\textsc{Degree of nilpotence with respect to uniform measures on balls.}
If $G$ is generated by a finite set $X$, one may naturally define the degree of nilpotence of $G$ using the sequence of measures $(\mu_n)_{n=1}^\infty$ defined by taking $\mu_n$ to be the uniform probability measure on the ball of radius $n$ in $G$ with respect to $X$. The main problem with adapting our results to this setting is that, in general, the sequence $\mu_n$ does not measure index uniformly (see the remark immediately after Theorem 1.12 of \cite{comm.prob}).
\begin{question}\label{qu:dc_k}
Let $G$ be a group with a finite generating set $X$, let $\mu_n$ be the uniform probability measure on the ball of radius $n$ in $G$ with respect to $X$, and write $M=(\mu_n)_{n=1}^\infty$. Suppose that $G$ is not virtually $k$-step nilpotent. Do we have $\dc^k_M(G)=0$?
\end{question}
If $M$ is defined using the uniform probability measures on the balls with respect to a finite generating set as in \cref{qu:dc_k}, and if $G$ is virtually nilpotent, then it is well known and easy to check that $M$ is an almost-invariant sequence of probability measures on $G$. It therefore follows from \cite[Theorem 1.12]{comm.prob} that $M$ measures index uniformly. In light of \cref{thm:dc_k}, a positive answer to \cref{qu:dc_k} would therefore extend both \cref{thm:dc_k,cor:indep} to the sequence of uniform probability measures on the balls with respect to a finite generating set.

\cref{qu:dc_k} seems to be difficult in general, although the answer is positive in some cases. For example, in \cref{yago} we present an argument that was communicated to us by Yago Antol\'in answering \cref{qu:dc_k} for hyperbolic groups (see \cref{hypnonilp}).

\bigskip\noindent\textsc{Acknowledgements.} The authors are grateful to Yago Antol\'in, Jack Button, Thiebout Delabie, Ben Green, Ana Khukhro, Ashot Minasyan, Aner Shalev and Alain Valette for helpful conversations, and to an anonymous referee for a careful reading of the manuscript and a number of helpful suggestions.

\section{The algebraic structure of probabilistically nilpotent groups} \label{sec:quant}

In this section we study the relation between $\dc_M^k(G)$ and the existence of finite-index $k$-step nilpotent subgroups of $G$, proving \cref{thm:dc_k}, \cref{prop:non-id} and \cref{prop:converse}. We start our proof of \cref{thm:dc_k} with the following version of \cite[Proposition 2.1]{comm.prob}, which was itself based on an argument of Neumann \cite{neumann}.

\begin{prop}\label{prop:large.centraliser}
Let $k\in\N$. Let $G$ be a group and let $M=(\mu_n)_{n=1}^\infty$ be a sequence of measures on $G$ that detects index uniformly at rate $\pi$. Let $\alpha\in(0,1]$, and suppose that $\dc_M^k(G)\ge\alpha$. Let $\gamma\in(0,1)$ be such that $\pi(\gamma)<\alpha$, and write
\[
X=\{(x_1,\ldots,x_k)\in G^k:[G:C_G([x_1,\ldots,x_k])]\le\textstyle\frac{1}{\gamma}\}.
\]
Then $\limsup_{n\to\infty}\mu_n(X)\ge\alpha-\pi(\gamma)$.
\end{prop}

\begin{proof}
By definition of $\dc_M$ there exists a sequence $n_1<n_2<\cdots$ such that $\dc_{\mu_{n_i}}^k(G)\ge\alpha-o(1)$. Writing $\E^{(n)}$ for expectation with respect to $\mu_n$, this means precisely that
\[
\E^{(n_i)}_{(x_1,\ldots,x_k)\in G^k}(\mu_{n_i}(C_G([x_1,\ldots,x_k])))\ge\alpha-o(1).
\]
Following Neumann \cite{neumann}, we note that therefore
\[
\begin{split}
\alpha\le\mu_{n_i}(X)\E^{(n_i)}_{(x_1,\ldots,x_k)\in X}(\mu_{n_i}(C_G([x_1,\ldots,x_k])))\qquad\qquad\qquad\qquad\qquad\qquad
       \\\qquad\qquad\qquad\qquad\qquad\qquad+\mu_{n_i}(G\backslash X)\E^{(n_i)}_{(x_1,\ldots,x_k)\in G^k\backslash X}(\mu_{n_i}(C_G([x_1,\ldots,x_k])))+o(1),
\end{split}
\]
and hence, by uniform detection of index,
\[
\alpha\le\mu_{n_i}(X)+\pi(\gamma)+o(1).
\]
The result follows.
\end{proof}

\begin{lemma}[{\cite[Proposition 1.1.1]{ls}}]\label{fin.gen.fin.ind}
Let $m,r\in\N$, and let $G$ be a group generated by $r$ elements. Then $G$ has at most $O_{m,r}(1)$ subgroups of index $m$.
\end{lemma}

\begin{proof}[Proof of \cref{thm:dc_k}]
We combine an induction used by Shalev in \cite[Proposition 2.2]{shalev} with the proof of \cite[Theorem 1.6]{comm.prob}. If $k=0$ then $\limsup_{n\to\infty}\mu_n(\{1\})\ge\alpha$, and so
the order of $G$ is $O_{\pi,\alpha}(1)$,
and the theorem holds. For $k>0$, let $\gamma=\frac{1}{2}\inf\{\beta\in(0,1]:\pi(\beta)\ge\frac{\alpha}{2}\}$, noting that therefore $\pi(\gamma)<\frac{\alpha}{2}$. \cref{prop:large.centraliser} therefore gives
\begin{equation}\label{eq:induction}
\limsup_{n\to\infty}\mu_n(\{(x_1,\ldots,x_k)\in G^k:[G:C_G([x_1,\ldots,x_k])]\le\textstyle\frac{1}{\gamma}\})\ge\frac{\alpha}{2}.
\end{equation}
Write $\Gamma$ for the intersection of all subgroups of $G$ of index at most $\frac{1}{\gamma}$, noting that $\Gamma$ is normal and
has index $O_{r,\pi,\alpha}(1)$
by \cref{fin.gen.fin.ind}. It follows from \eqref{eq:induction} that
\[
\limsup_{n\to\infty}\mu_n(\{(x_1,\ldots,x_k)\in G^k:[x_1,\ldots,x_k]\in C_G(\Gamma)\})\ge\textstyle\frac{\alpha}{2},
\]
or equivalently that
\[
\dc_M^{k-1}(G/C_G(\Gamma))\ge\textstyle\frac{\alpha}{2}.
\]
By induction, $G/C_G(\Gamma)$ has a $(k-1)$-step nilpotent subgroup $N_0$ of index at most $O_{r,\pi,k,\alpha}(1)$. Writing $N$ for the pullback of $N_0$ to $G$, the intersection $N\cap\Gamma$ is $k$-step nilpotent and has index at most $O_{r,\pi,k,\alpha}(1)$ in $G$, and so the theorem is proved.
\end{proof}

\cref{prop:non-id} is essentially based on the following lemma.
\begin{lemma}\label{lem:coset}
Let $G$ be a group, and $u,g\in G$. Then $\{x\in G:[u,x]=g\}$ is either empty or a coset of $C_G(u)$.
\end{lemma}
\begin{proof}
If  $\{x\in G:[u,x]=g\}$ is not empty then $[u,x_0]=g$ for some $x_0\in G$, in which case we have $\{x\in G:[u,x]=g\}=\{x\in G:u^x=u^{x_0}\}=C_G(u)x_0$.
\end{proof}

The point of the following lemma is that sequences of measures that measure index uniformly give the same measure to right-cosets of a subgroup that they give to left-cosets of that subgroup.
\begin{lemma}\label{lem:left/right}
Let $M=(\mu_n)_{n=1}^\infty$ be a sequence of measures that measures index uniformly. Then $\mu_n(Hx)\to1/[G:H]$ uniformly over all $x\in G$ and all subgroups $H$ of $G$.
\end{lemma}
\begin{proof}
This is because $Hx=x(H^x)$ and $H^x$ has the same index as $H$.
\end{proof}

\begin{proof}[Proof of \cref{prop:non-id}]
\cref{lem:coset,lem:left/right} and the definition of uniform measurement of index imply that for every $x_1,\ldots,x_k\in G$ we have either
\begin{align*}
\mu_{n}(\{x\in G:[[x_1,\ldots,x_k],x]=g\})&\to\frac{1}{[G:C_G([x_1,\ldots,x_k])]}\\
    &=\lim_{n\to\infty}\mu_n(C_G([x_1,\ldots,x_k]))
\end{align*}
or $\mu_{n}(\{x\in G:[[x_1,\ldots,x_k],x]=g\})\to0$, and that this convergence is uniform over all $x_1,\ldots,x_k$. Writing $\E^{(n)}$ for expectation with respect to $\mu_n$, it follows that
\begin{align*}
\limsup_{n\to\infty}P_{\mu_n}^k(G,g)&=\limsup_{n\to\infty}\E^{(n)}_{(x_1,\ldots,x_k)\in G^k}(\mu_{n}(\{x\in G:[[x_1,\ldots,x_k],x]=g\}))\\
    &\le\limsup_{n\to\infty}\E^{(n)}_{(x_1,\ldots,x_k)\in G^k}(\mu_{n}(C_G([x_1,\ldots,x_k])))\\
    &=P_M^k(G,1),
\end{align*}
as required.
\end{proof}

We close this section by proving \cref{prop:converse}.

\begin{proof}[Proof of \cref{prop:converse}]
We use a similar argument to that of \cite[Proposition 1.15]{comm.prob}. Fix elements $x_1,\ldots,x_k\in\Gamma$. The fact that $\Gamma/H$ is $k$-step nilpotent implies that $[[x_1,\ldots,x_k],y] \in H$ for every $y\in\Gamma$. Since $|H| \leq d$, it follows from \cref{lem:coset} that
$C_\Gamma([x_1,\ldots,x_k])$ has index at most $d$ in $\Gamma$, and hence at most $dm$ in $G$. Since this holds for every $x_1,\ldots,x_k\in\Gamma$, the result follows by uniform measurement of index.
\end{proof}

\section{Products of measures that measure index uniformly}\label{sec:prod.meas}
The aim of this short section is to prove the following result.
\begin{prop}\label{prop:fin.quot.=>seq}
Let $(\mu_n)_{n=1}^\infty$ be a sequence of measures that measure index uniformly on a group $G$, and suppose that $V\subset G^k$ is negligible by finite quotients of $G^k$. Then $\mu_n(V)\to0$ as $n\to\infty$.
\end{prop}

\begin{remark}\label{rem:fin.quot.=>seq}
To see that being negligible by finite quotients is \emph{strictly} stronger than having zero density with respect to a sequence of measures measuring index uniformly, consider the example in which $\mu_n$ is the uniform probability measure on the F\o lner set $\{-n,\ldots,n\}\subset\Z$, and the set $A$ is defined as
\[
A=\bigcup_{k=1}^\infty\{2^k+1,\ldots,2^k+k\}.
\]
Then $A$ satisfies $\mu_n(A)\to0$ as $n\to\infty$, but is not negligible by finite quotients of $\Z$.
\end{remark}

The proof of \cref{prop:fin.quot.=>seq} rests on the following lemma.
\begin{lemma}\label{lem:prod.meas}
Let $G_1,\ldots,G_k$ be groups, and for each $i$ let $(\mu_n^{(i)})_{n=1}^\infty$ be a sequence of measures such that for every $x_i\in G_i$ and every finite-index subgroup $H_i<G_i$ we have $\mu^{(i)}_n(x_iH_i)\to1/[G_i:H_i]$ as $n\to\infty$. Then
\[
\mu_n^{(1)}\times\cdots\times\mu_n^{(k)}(xH)\to\frac{1}{[G_1\times\cdots\times G_k:H]}
\]
for every finite-index subgroup $H<G_1\times\cdots\times G_k$ and every element $x\in G_1\times\cdots\times G_k$.
\end{lemma}
\begin{proof}
Note first that
\[
\mu_n^{(1)}\times\cdots\times\mu_n^{(k)}\left(x\textstyle\prod_{i=1}^kH_i\right)\to\frac{1}{[\prod_{i=1}^kG_i:\prod_{i=1}^kH_i]}
\]
for all $x\in\prod_{i=1}^kG_i$ and all finite-index subgroups of the form $\prod_{i=1}^kH_i<\prod_{i=1}^kG_i$ with $H_i<G_i$ for each $i$. It therefore suffices to show that an arbitrary finite-index subgroup $H<\prod_{i=1}^kG_i$ has a finite-index subgroup of the form $\prod_{i=1}^kH_i$. However, if $H$ has index $d$ in $G$ then $H\cap G_i$ has index at most $d$ in $G_i$ for each $i$. The product subgroup $\prod_{i=1}^k(H\cap G_i)$ therefore has index at most $d^k$ in $\prod_{i=1}^kG_i$, and hence in $H$, as required.
\end{proof}
\begin{proof}[Proof of \cref{prop:fin.quot.=>seq}]
Let $\eps>0$, and let $H\lhd G^k$ be a finite-index normal subgroup such that $|VH/H|\le\frac{1}{2}\eps[G^k:H]$. As $G^k$ contains only finitely many cosets of $H$, it follows from \cref{lem:prod.meas} that there exists $N\in\N$ such that for every $x\in G^k$ and every $n\ge N$ we have $\mu_n(xH)\le2/[G^k:H]$, and hence $\mu_n(V)\le\mu_n(VH)\le\eps$. We therefore have $\mu_n(V)\to0$, as required.
\end{proof}

\section{Equations over virtually nilpotent groups in terms of polynomial mappings}\label{sec:v.nilp}
In this section and the next we prove \cref{thm:v.nilp.indep}. An important tool in the proof is the notion of a \emph{polynomial mapping} of a group. These have been studied extensively by Leibman \cite{leibman}, and have found applications to finding prime solutions to linear systems of equations \cite{gt} and to the study of harmonic functions on groups \cite{mpty}. They are defined as follows.
\begin{definition}[derivatives and polynomial mappings]
Let $G,H$ be groups and let $\varphi:G\to H$. Given $u\in G$, we define the \emph{$u$-derivative} $\der_u\varphi:G\to H$ of $\varphi$ via $\der_u\varphi(x)=\varphi(x)^{-1}\varphi(xu)$. Given $d\in\N$, we say that $\varphi$ is \emph{polynomial of degree $d$} if $\der_{u_1}\cdots\der_{u_{d+1}}\varphi\equiv1$ for all $u_1,\ldots,u_{d+1} \in G$.
\end{definition}
\begin{remark*}Leibman actually defines the more refined notion of being polynomial \emph{relative to a generating set $S$} for $G$; the above definition corresponds to being polynomial relative to $G$. Nonetheless, in the present paper the range of every mapping we consider will be nilpotent, and Leibman shows that a mapping of $G$ into a nilpotent group is polynomial relative to some generating set for $G$ if and only if it is polynomial relative to $G$ \cite[Proposition 3.5]{leibman}, so we lose no generality by using the definition above.
\end{remark*}

The basic scheme of the proof of \cref{thm:v.nilp.indep} is to show that equations over virtually nilpotent groups are polynomial mappings, so that the set of solutions to an arbitrary equation can be viewed as the set of roots of some polynomial. We do that in this section. The idea is then to use the familiar notion that the set of roots of a polynomial is `sparse' in some sense; we do this in the next section.

To state the main result of this section requires some notation.
Let $G$ be a group, let $H\lhd G$ be a normal subgroup, let $\varphi\in F_k\ast G$ be an equation over $G$ and let $g\in G^k$. Given $h\in H^k$, note that $\varphi(hg)\in H\varphi(g)$, so that we may define a mapping $\varphi_{H,g}:H^k\to H$ via
\[
\varphi_{H,g}(h)=\varphi(hg)\varphi(g)^{-1}.
\]
We may then describe the set of solutions to $\varphi=1$ in the coset $H^kg$ as
\begin{equation}\label{eq:solutions.coset}
G_\varphi\cap H^kg=\{h\in H^k:\varphi_{H,g}(h)=\varphi(g)^{-1}\}g.
\end{equation}
Our result is then as follows.
\begin{prop}\label{lem:phi_t.poly}
Let $G$ be a group, let $N\lhd G$ be a nilpotent normal subgroup, let $\varphi\in F_k\ast G$ be an equation over $G$, and let $g\in G^k$. Then the map $\varphi_{N,g}:N^k\to N$ is polynomial.
\end{prop}
In the case $N=G$, \cref{lem:phi_t.poly} is a consequence of the following result of Leibman.
\begin{theorem}[Leibman {\cite[Theorem 3.2]{leibman}}]\label{thm:leib.grp}
If $H$ is a group and $N$ is a nilpotent group then the polynomial mappings $H\to N$ form a group under the operations of taking pointwise products and pointwise inverses.
\end{theorem}
Here, if $\varphi,\psi:H\to N$ are two polynomial mappings into a nilpotent group $N$ then the \emph{pointwise product} $\varphi\psi:H\to N$ is defined by setting $(\varphi\psi)(h)=\varphi(h)\psi(h)$, and the \emph{pointwise inverse} $\varphi^{(-1)}:H\to N$ is defined by setting $\varphi^{(-1)}(h)=\varphi(h)^{-1}$. Since constant maps $G^k \to G$ are trivially polynomial of degree $0$, and the maps $G^k \to G$ sending $(x_1,\ldots,x_k)$ to $x_i$ are trivially polynomial of degree $1$, it follows immediately from \cref{thm:leib.grp} that an equation over a nilpotent group $G$ is a polynomial $G^k\to G$. Our proof of \cref{lem:phi_t.poly} consists of reducing to the special case of \cref{thm:leib.grp}.
\begin{proof}[Proof of \cref{lem:phi_t.poly}]
We can view the equation $\varphi$ as a concatenation of \emph{variables} $x_i^{\pm 1} \in F_k$ and \emph{constants} $c \in G$. Write $g=(g_1,\ldots,g_k)$, and let $h = (h_1,\ldots,h_k)\in N^k$. Moving the elements $g_i^{\pm1} \in G$ and constants $c \in G$ one by one to the right of the word $\varphi(h_1g_1,\ldots,h_kg_k)$, conjugating the elements $h_i^{\pm1}$ as we go, we see that $\varphi_{N,g}(h)$ is a product of elements of the form $(h_i^{\pm1})^x$, with $x\in G$ depending only on $g_1,\ldots,g_k$, not on $h_1,\ldots,h_k$. Given any fixed $x\in G$ the maps $N^k\to N$ defined by $(h_1,\ldots,h_k)\mapsto h_i^x$ are polynomial of degree $1$, and so \cref{thm:leib.grp} implies that $\varphi_{N,g}$ is polynomial, as required.
\end{proof}

\section{Sparsity of roots of polynomial mappings}\label{sec:polys}
In this section we show that the set of roots of a polynomial mapping into a torsion-free nilpotent group is negligible by finite quotients, as follows.
\begin{definition}[closed subgroup]
A subgroup $\Gamma$ of a group $G$ is said to be \emph{closed} in $G$ if for every $x\in G$ and $n\in\Z$ we have $x^n\in\Gamma$ if and only if $x\in\Gamma$.
\end{definition}
\begin{theorem}\label{thm:leibman}
Let $G$ be a finitely generated group, let $N$ be a nilpotent group with a closed subgroup $\Gamma$, and let $\varphi:G\to N$ be polynomial. Then for every $x\in N$ such that $\varphi(G)\not\subset x\Gamma$ the set $\varphi^{-1}(x\Gamma)$ is negligible by finite quotients of $G$.
\end{theorem}
This is in the same spirit as the following theorem of Leibman.
\begin{theorem}[Leibman {\cite[Proposition 4.3]{leibman}}]\label{thm:leibman.orig}
Let $G$ be a countable amenable group, and let $(\mu_n)_{n=1}^\infty$ be the sequence of uniform probability measures on some F\o lner sequence on $G$. Let $N$ be a nilpotent group, let $\Gamma$ be a closed subgroup of $N$, and let $\varphi:G\to N$ be polynomial. Then for every $x\in N$ such that $\varphi(G)\not\subset x\Gamma$ we have $\mu_n(\varphi^{-1}(x\Gamma))\to0$ as $n\to\infty$.
\end{theorem}
\begin{remark*}By \cite[Theorem 1.12]{comm.prob}, \cref{prop:fin.quot.=>seq,rem:fin.quot.=>seq}, \cref{thm:leibman} is strictly stronger than \cref{thm:leibman.orig} in the case where $G$ is finitely generated. It also applies to non-amenable $G$, unlike \cref{thm:leibman.orig}. On the other hand, \cref{thm:leibman.orig} does not require $G$ to be finitely generated, and it follows from \cite[Proposition 3.21]{leibman} that every polynomial mapping of $G$ into a nilpotent group factors through some amenable quotient of $G$, so \cref{thm:leibman.orig} does have some implicit content even when $G$ is not amenable.
\end{remark*}

We divide the proof of \cref{thm:leibman} into two parts. The first part reduces to the case where $N=\Z$, as follows.
\begin{prop}\label{prop:reduc.to.Z}
Let $G$ be a finitely generated group, let $N$ be a nilpotent group with a closed subgroup $\Gamma$, let $\varphi:G\to N$ be polynomial, and let $x\in N$ be such that $\varphi(G)\not\subset x\Gamma$. Then there is a non-constant polynomial mapping $\psi:G\to\Z$ such that $\varphi^{-1}(x\Gamma)\subset\psi^{-1}(0)$.
\end{prop}
The second part proves the theorem in this case, as follows.
\begin{prop}\label{prop:leibman.Z}
Let $G$ be a finitely generated group and let $\varphi:G\to\Z$ be a non-constant polynomial mapping. Then $\varphi^{-1}(0)$ is negligible by finite quotients.
\end{prop}

In proving \cref{prop:reduc.to.Z} we use the following characterisation of closed subgroups of nilpotent groups.

\begin{prop}[Bergelson--Leibman {\cite[Proposition 1.19]{bl}}]\label{prop:closed}
Let $N$ be a finitely generated nilpotent group. Then a subgroup $\Gamma<N$ is closed in $N$ if and only if there exists a series $\Gamma=\Gamma_0\lhd\Gamma_1\lhd\ldots\lhd\Gamma_r=N$ with $\Gamma_i/\Gamma_{i-1}\cong\Z$ for every $i$.
\end{prop}

We also use the following trivial lemma.

\begin{lemma}[Leibman {\cite[Proposition 1.10]{leibman}}]\label{lem:comp.w/hom}
Let $G$, $H$ and $H'$ be groups, let $\varphi:G\to H$ be polynomial of degree $d$, and let $\pi:H\to H'$ be a homomorphism. Then the composition $\pi\circ\varphi:G\to H'$ is also polynomial of degree $d$.
\end{lemma}

\begin{proof}[Proof of \cref{prop:reduc.to.Z}]
Since identity map is polynomial of degree $1$, it follows from \cref{thm:leib.grp} that $x^{-1}\varphi$ is polynomial. Since $(x^{-1}\varphi)(g)\in\Gamma$ precisely when $\varphi(g)\in x\Gamma$, upon replacing $\varphi$ by $x^{-1}\varphi$ we may therefore assume that $x=1$.

Let $\Gamma=\Gamma_0\lhd\Gamma_1\lhd\ldots\lhd\Gamma_r=N$ be the series given by \cref{prop:closed}, so that $\Gamma_i/\Gamma_{i-1}\cong\Z$ for every $i$, and let $k$ be minimal such that $\varphi(G)\subset\Gamma_k$, noting that $k\ge1$ by assumption. Write $\pi:\Gamma_k\to\Gamma_k/\Gamma_{k-1}\cong\Z$ for the quotient homomorphism, and define $\psi=\pi\circ\varphi:G\to\Gamma_k/\Gamma_{k-1}\cong\Z$, noting that $\psi$ is polynomial by \cref{lem:comp.w/hom}. We then have $\psi^{-1}(0)=\varphi^{-1}(\Gamma_{k-1})\supset\varphi^{-1}(\Gamma)$, as required.
\end{proof}

The first step in our proof of \cref{prop:leibman.Z} is to reduce to the case where $G$ is torsion-free nilpotent, via the following result of Meyerovitch, Perl, Yadin and the second author \cite{mpty}.
\begin{lemma}\label{lem:tf.quot}
Let $G$ be a group, and let $\varphi:G\to\Z$ be a polynomial mapping of degree $d$. Then there is a torsion-free $d$-step nilpotent quotient $G'$ of $G$ and a polynomial mapping $\hat\varphi:G'\to\Z$ of degree $d$ such that, writing $\pi:G\to G'$ for the quotient homomorphism, we have $\varphi=\hat\varphi\circ\pi$.
\end{lemma}
\begin{proof}
This is immediate from \cite[Lemmas 2.5 \& 4.4]{mpty}.
\end{proof}
In fact, although \cref{lem:tf.quot} is sufficient for our purposes in the present paper, in \cref{sec:polymapnilp} we take the opportunity to deduce from it a similar result for polynomial mappings into arbitrary torsion-free nilpotent groups.

An important benefit of \cref{lem:tf.quot} is that it allows us in the proof of \cref{prop:leibman.Z} to exploit the existence of certain coordinate systems on torsion-free nilpotent groups. We give a basic description of coordinate systems here; see \cite[3.8--3.19]{leibman} for a more detailed description of coordinate systems and their relationship to polynomial mappings, and \cite[\S4.2]{mpty} for details on a particularly natural coordinate system to use when studying polynomial mappings to nilpotent groups.

Given a finitely generated torsion-free nilpotent group $G$, there exists a central series $\{1\}=G_0\lhd G_1\lhd\ldots\lhd G_m = G$ such that $G_i/G_{i-1}\cong\Z$ for every $i$. Picking $e_i\in G_i$ for each $i$ in such a way that $G_{i-1}e_i$ is a generator for $G_i/G_{i-1}$, every element $g\in G$ then has a unique expression
\begin{equation}\label{eq:coords}
g=e_1^{v_1}\cdots e_m^{v_m}
\end{equation}
for some $v_1,\ldots,v_m\in\Z$; we call $(e_1,\ldots,e_m)$ a \emph{basis} for $G$. We call the $v_i$ in the expression \eqref{eq:coords} the \emph{coordinates} of $g$ with respect to $(e_1,\ldots,e_m)$, and call the map $G\to\Z^m$ taking an element of $G$ to its coordinates the \emph{coordinate mapping} of $G$ with respect to $(e_1,\ldots,e_m)$. We often abbreviate the expression $e_1^{v_1}\cdots e_m^{v_m}$ as $\bfe^\bfv$.

\begin{prop}[Leibman {\cite[Proposition 3.12]{leibman}}]\label{prop:coord.poly}
Let $G$ and $N$ be finitely generated torsion-free nilpotent groups with bases $(e_1,\ldots,e_m)$ and $(f_1,\ldots,f_n)$, respectively, and let $\alpha:G\to\Z^m$ and $\beta:N\to\Z^n$ be the corresponding coordinate mappings. Then a mapping $\varphi:G\to N$ is polynomial if and only if $\beta\circ\varphi\circ\alpha^{-1}:\Z^m\to\Z^n$ is polynomial.
\end{prop}

Polynomial mappings $\Z^m\to\Z^n$ are just standard polynomials in $m$ variables, although we caution, as Leibman does in \cite[1.8]{leibman}, that these polynomials can have non-integer rational coefficients: the polynomial $\frac{1}{2}n^2+\frac{1}{2}n$ maps $\Z\to\Z$, for example.

Leibman \cite[Corollary 3.7]{leibman} shows that in a nilpotent group $G$ the operations of multiplication $G\times G\to G$ defined by $(g_1,g_2)\mapsto g_1g_2$, and raising to a power $G\times\Z\to G$ defined by $(g,n)\mapsto g^n$, are polynomial mappings. Given a finitely generated torsion-free nilpotent group $G$ with basis $(e_1,\ldots,e_m)$, it therefore follows from \cref{prop:coord.poly} that there exist polynomials $\mu_1,\ldots,\mu_m:\Z^{2m}\to\Z$ and $\epsilon_1,\ldots,\epsilon_m:\Z^{m+1} \to \Z$ such that
\begin{equation}\label{eq:coord.mult}
\bfe^\bfv \cdot \bfe^\bfw = e_1^{\mu_1(\bfv,\bfw)} \cdots e_m^{\mu_m(\bfv,\bfw)}
\end{equation}
and
\begin{equation}\label{eq:coord.power}
(\bfe^\bfv)^n = e_1^{\epsilon_1(\bfv,n)} \cdots e_m^{\epsilon_m(\bfv,n)}
\end{equation}
for every $\bfv,\bfw\in\Z^m$ and every $n\in\Z$. Leibman notes this in \cite[Corollary 3.13]{leibman}. It recovers a result of Hall \cite[Theorem 6.5]{hall}.

In light of \cref{prop:coord.poly}, if $G$ is torsion-free nilpotent then \cref{prop:leibman.Z} follows from the following result.

\begin{prop} \label{prop:negl}
Let $G$ be a torsion-free nilpotent group with basis $(e_1,\ldots,e_m)$, and let $\alpha:G\to\Z^m$ be the corresponding coordinate mapping. Let $p: \Z^m \to \Z$ be a non-zero polynomial, and let
\[
\mathcal{N}_p := \{ g\in G : p\circ\alpha(g) = 0 \}.
\]
Then $\mathcal{N}_p$ is negligible in $G$ by finite quotients.
\end{prop}

The first step in the proof of \cref{prop:negl} is to construct the quotients that we will use to show that $\mathcal{N}_p$ is negligible by finite quotients. Given a group $G$ we write $G^{(n)}$ is the subgroup generated by all $n^{\text{th}}$ powers of elements of $G$, and $G(n)$ for the quotient $G/G^{(n)}$. If $G$ is finitely generated and torsion-free nilpotent with basis $(e_1,\ldots,e_m)$ then we write $G_i(n)$ for the image of $G_i$ under the quotient map $G \to G(n)$. The precise statement that we prove in order to deduce \cref{prop:negl} is then as follows.
\begin{prop}\label{prop:negl.precise}
Let $G$ be a torsion-free nilpotent group with basis $(e_1,\ldots,e_m)$, and let $\alpha:G\to\Z^m$ be the corresponding coordinate mapping. Let $p: \Z^m \to \Z$ be a non-zero polynomial, and let
\[
\mathcal{N}_p := \{ g\in G : p\circ\alpha(g) = 0 \}.
\]
Then 
\begin{equation}\label{eq:Np.negl}
\frac{|\mathcal{N}_pG^{(n)}/G^{(n)}|}{|G/G^{(n)}|}\to0
\end{equation}
as $n\to\infty$ through the primes.
\end{prop}
\begin{remark}
An inspection of the argument shows that there exists an integer $n_0=n_0(G,e_1,\ldots,e_m)$ such that \eqref{eq:Np.negl} holds as $n\to\infty$ through those positive integers coprime to $n_0$.
\end{remark}

\begin{lemma}\label{lem:finite}
Let $G$ be a finitely generated torsion-free nilpotent group with basis $(e_1,\ldots,e_m)$. Then there exists an integer $n_0 = n_0(G,e_1,\ldots,e_m)$ such that for every positive integer $n$ coprime to $n_0$ and every $i=1,\ldots,m$ we have $G_i(n)/G_{i-1}(n)\cong C_n$.
\end{lemma}

\begin{proof}
Pick $n_0$ so that the coefficients of the polynomials $\mu_i,\epsilon_i$ given in \eqref{eq:coord.mult} and \eqref{eq:coord.power} all lie in $\frac{1}{n_0}\Z$. Fix $n$ coprime to $n_0$, and write $\Phi_n: G \to G(n)$ for the quotient homomorphism.

Note that $\epsilon_i(\bfv,0) = 0$ for each $i$ and each $\bfv \in \Z^m$, so that the polynomials $\epsilon_i(\bfv,-): \Z \to \Z$ have no constant term. By the definition of $n_0$, for each $i$ and each $\bfv \in \Z^m$ we have $\epsilon_i(\bfv,n) = c_{i,\bfv,n}n/n_0\in \Z$ for some $c_{i,\bfv,n} \in \Z$. As $n$ is coprime to $n_0$, it follows that $n$ divides $\epsilon_i(\bfv,n)$, and so
\begin{equation}\label{eq:coords.Cn}
G^{(n)} = \langle e_1^n,\ldots,e_m^n \rangle.
\end{equation}
The polynomials $\mu_i: \Z^{2m} \to \Z$ similarly have no constant term, and so there exist polynomials $\bar\mu_1,\ldots,\bar\mu_m: \Z^{2m} \to \Z$ such that
\begin{equation} \label{eqn:multn}
\bfe^{\bfv n} \cdot \bfe^{\bfw n} = e_1^{\bar\mu_1(\bfv,\bfw)n} \cdots e_m^{\bar\mu_m(\bfv,\bfw)n}.
\end{equation}
By \eqref{eq:coords.Cn} and successive applications of \eqref{eqn:multn}, it therefore follows that
\begin{equation}\label{eq:G^(n)}
G^{(n)}=\{\bfe^{\bfv n}:\bfv\in\Z^m\}.
\end{equation}
It is clear that $G_i(n)/G_{i-1}(n)$ is generated by $G_{i-1}(n)\Phi_n(e_i)$ and is a quotient of $C_n$, so it is enough to show that $\Phi_n(e_i^r) \notin G_{i-1}(n)$ for $r=1,\ldots,n-1$. If, on the contrary, $\Phi_n(e_i^r) \in G_{i-1}(n)$ for some such $r$, then if would follow from \eqref{eq:G^(n)} that
\[
e_i^r = \bfe^{\bfv n} \cdot e_1^{w_1} \cdots e_{i-1}^{w_{i-1}}
\]
for some $\bfv \in \Z^m$ and some $w_1,\ldots,w_{i-1} \in \Z$, which would give
\[
e_i^{v_in-r} e_{i+1}^{v_{i+1}n} \cdots e_m^{v_mn} \in G_{i-1}.
\]
Since $v_in-r \neq 0$ whenever $1 \leq r \leq n-1$, this would contradict the uniqueness of coordinates, and so we indeed have $\Phi_n(e_i^r) \notin G_{i-1}(n)$, as required.
\end{proof}

\begin{remark}
Note that the conclusion of \cref{lem:finite} does not necessarily hold for an arbitrary $n\in\N$. For instance, if $G = \left(\begin{smallmatrix}1&\Z&\Z\\0&1&\Z\\0&0&1\end{smallmatrix}\right)$ is the integral Heisenberg group then $m=3$ and $G_1(2)/G_0(2)$ is trivial.
\end{remark}

\begin{proof}[Proof of \cref{prop:negl.precise}]
Throughout this proof $n$ is a prime. Write $d$ for the degree of $p$, and for each prime $n$ write $\Phi_n: G \to G(n)$ for the quotient homomorphism. By \cref{lem:finite} the desired conclusion \eqref{eq:Np.negl} is equivalent to the statement that
\begin{equation}\label{eq:Np.negl'}
\frac{|\Phi_n(\mathcal{N}_p)|}{n^m}\to0
\end{equation}
as $n\to\infty$ through the primes. If $m = 1$ then $|\mathcal{N}_p|\le d$, and so for every $n$ we also have $|\Phi_n(\mathcal{N}_p)|\le d$, which certainly implies \eqref{eq:Np.negl'}. We may therefore assume that $m\ge2$ and proceed by induction on $m$.

We can view $p$ as an element of $\Q[X_1,\ldots,X_m]$. If $p \in \Q[X_2,\ldots,X_m]$ then it follows by applying the induction hypothesis to $G/G_1$ that
\[
\frac{|\Phi_n(\mathcal{N}_p)G_1(n)/G_1(n)|}{n^{m-1}}\to0
\]
as $n\to\infty$, which implies \eqref{eq:Np.negl'} by \cref{lem:finite}. We may therefore assume that
\[
p(X_1,\ldots,X_m) = \sum_{i=0}^d X_1^i p_i(X_2,\ldots,X_m)
\]
for some $p_0,\ldots,p_d \in \Q[X_2,\ldots,X_m]$ with $p_j \neq 0$ for some $j \geq 1$. Writing
\[
\mathcal{P} = \{ g \in G : p_j(\alpha_2(g),\ldots,\alpha_m(g)) = 0 \},
\]
we have
\[
\frac{|\Phi_n(\mathcal{P})/G_1(n)|}{n^{m-1}}\to0
\]
as $n\to\infty$ by induction, and hence
\begin{equation}\label{eq:P.negl}
\frac{|\Phi_n(\mathcal{P})|}{n^m}\to0.
\end{equation}
For $g\in\mathcal{N}_p\setminus\mathcal{P}$, on the other hand, $\alpha_1(g)$ is a root of the non-zero polynomial
\[
p(X,\alpha_2(g),\ldots,\alpha_m(g)) \in\Q[X]
\]
of degree at most $d$, and so $|(\mathcal{N}_p\setminus\mathcal{P}) \cap G_1x| \leq d$ for all $x \in G$. By taking images in $G(n)$ this implies that $|\Phi_n(\mathcal{N}_p\setminus\mathcal{P}) \cap G_1(n)x| \leq d$ for all $x \in G(n)$, and so
\[
|\Phi_n(\mathcal{N}_p\setminus\mathcal{P})| \leq d|G(n)/G_1(n)| = dn^{m-1}
\]
for every large enough prime $n$. Combined with \eqref{eq:P.negl}, this implies \eqref{eq:Np.negl'}, as required.
\end{proof}

\begin{proof}[Proof of \cref{prop:leibman.Z}]
Let $G'$, $\pi$ and $\hat\varphi$ be as given by \cref{lem:tf.quot}. It follows from Propositions \ref{prop:coord.poly} and \ref{prop:negl.precise} that $\hat\varphi^{-1}(0)$ is negligible by finite quotients of $G'$, and hence that $\varphi^{-1}(0)=\pi^{-1}(\hat\varphi^{-1}(0))$ is negligible by finite quotients of $G$, as required.
\end{proof}

\begin{lemma}\label{lem:coset.negl}
Let $G$ be a group, let $H \lhd G$ be a finite-index normal subgroup, and let $g\in G$. Let $V\subset H$ be negligible by finite quotients in $H$. Then $Vg$ is negligible by finite quotients in $G$.
\end{lemma}

\begin{proof}
Let $\varepsilon > 0$. Then there exists a normal subgroup $K \lhd H$ of finite index such that $|VK / K| \le\varepsilon|G/K|$. Since $K$ has finite index in $G$, there exists a finite-index subgroup $L < K$ such that $L \lhd G$, and then we have $|VgL / L| = |VL / L| \le |VK / K| |K / L| \leq \varepsilon |G/K| |K/L| = \varepsilon |G / L|$.
\end{proof}

\begin{proof}[Proof of \cref{thm:v.nilp.indep}]
Since $N^k$ has finite index in $G^k$, the theorem may be restated as saying that for every $g\in G^k$ with $G_\varphi\cap N^kg\ne N^kg$ we have $G_\varphi\cap N^kg$ negligible by finite quotients of $G^k$. This follows readily from \eqref{eq:solutions.coset}, \cref{lem:phi_t.poly}, \cref{thm:leibman,lem:coset.negl}.
\end{proof}

\section{Finite quotients} \label{sec:gallagher}

In this section we prove \cref{thm:gallagher} and deduce \cref{thm:res.fin.} from it. We fix $k \in \N$ throughout. In \cref{ssec:sketchy-sketch,ssec:sketch,ssec:defns,ssec:details} we also fix a finite group $G$ and a normal subgroup $N \lhd G$, and define
\[ \mathcal{N}_k(G) := \{ (x_1,\ldots,x_{k+1}) \in G^{k+1} \mid [x_1,\ldots,x_{k+1}] = 1 \}. \]
Note that \cref{thm:gallagher} is equivalent to the following result.
\begin{theorem} \label{thm:gallagher2}
We have $|\mathcal{N}_k(G)| \leq |\mathcal{N}_k(N)| \times |\mathcal{N}_k(G/N)|$.
\end{theorem}
We prove \cref{thm:gallagher2} in \cref{ssec:sketchy-sketch,ssec:sketch,ssec:defns,ssec:details}. In \cref{ssec:res.fin} we prove \cref{thm:res.fin.}.

\subsection{Submultiplicativity of degree of nilpotence} \label{ssec:sketchy-sketch}

Here we sketch the proof of \cref{thm:gallagher2}, omitting the proof of a technical result---\cref{prop:gallagher}---that we give in \cref{ssec:sketch,ssec:defns,ssec:details}.

For subsets $A_1,\ldots,A_{k+1} \subseteq G$, define
\[ f_k(A_1,\ldots,A_{k+1}) = | \mathcal{N}_k(G) \cap (A_1 \times \cdots \times A_{k+1}) |. \]
If $A_i = \{a_i\}$ is a singleton for some $i$, for simplicity of notation we will write $f_k(\ldots,\{a_i\},\ldots)$ as $f_k(\ldots,a_i,\ldots)$.

Given cosets $x_1N,\ldots,x_{k+1}N \in G/N$, it is clear that if $f_k(x_1N,\ldots,x_{k+1}N) \neq 0$ then the element $[x_1N,\ldots,x_{k+1}N]$ is trivial in $G/N$. Thus the number of elements $(x_1N,\ldots,x_{k+1}N) \in (G/N)^{k+1}$ with $f_k(x_1N,\ldots,x_{k+1}N) \neq 0$ is at most $|\mathcal{N}_k(G/N)|$, and we obtain
\[ |\mathcal{N}_k(G)| \leq |\mathcal{N}_k(G/N)| \times \max \{ f_k(x_1N,\ldots,x_{k+1}N) \mid (x_1N,\ldots,x_{k+1}N) \in (G/N)^{k+1} \}. \]
Since $f_k(N,\ldots,N) = |\mathcal{N}_k(N)|$, \cref{thm:gallagher2} therefore follows from the following Lemma.

\begin{lemma} \label{lem:frN}
For every $(x_1N,\ldots,x_{k+1}N) \in (G/N)^{k+1}$ we have
\[ f_k(x_1N,\ldots,x_{k+1}N) \leq f_k(N,\ldots,N). \]
\end{lemma}

The proof of \cref{lem:frN} uses the following proposition, to be proved in \cref{ssec:sketch,ssec:defns,ssec:details}.
\begin{prop} \label{prop:gallagher}
For any $g \in G$ and $xN \in G/N$, we have 
\[ f_k(xN,g,N,N,\ldots,N) \leq f_k(N,g,N,N,\ldots,N). \]
\end{prop}

\begin{proof}[Proof of \cref{lem:frN}]
We will show that for each $i \in \{ 0,\ldots,k \}$, we have
\[ f_k(x_1N,\ldots,x_iN,x_{i+1}N,N,\ldots,N) \leq f_k(x_1N,\ldots,x_iN,N,N,\ldots,N), \]
which will imply the result. Note that for $i = 0$, this follows immediately from \cref{prop:gallagher} (by summing over $g \in N$), hence we can without loss of generality assume that $i \geq 1$.

Let $s \in \mathbb{N}$. Since $[y^{-1},g] = [g,y]^{y^{-1}}$ for all $y,g \in G$, we have an identity
\[ [(xn)^{-1},g,n_3,\ldots,n_{s+1}] = [g,xn,n_3^{xn},\ldots,n_{s+1}^{xn}]^{(xn)^{-1}}. \]
As $N$ is normal in $G$, this induces a bijection
\begin{align*}
\mathcal{N}_s(G) \cap \left(x^{-1}N \times \{g\} \times N^{s-1}\right) &\leftrightarrow \mathcal{N}_s(G) \cap \left(\{g\} \times xN \times N^{s-1}\right),
\end{align*}
and so we have
\[ f_s(x^{-1}N,g,N,\ldots,N) = f_s(g,xN,N,\ldots,N) \]
for all $g \in G$ and $xN \in G/N$. Thus \cref{prop:gallagher} implies that
\[ f_s(g,xN,N,\ldots,N) = f_s(x^{-1}N,g,N,\ldots,N) \leq f_s(N,g,N,\ldots,N) = f_s(g,N,N,\ldots,N). \]
Now fix $i \in \{ 1,\ldots,k \}$. Then this last inequality implies
\begin{align*}
f_k(x_1N,\ldots,x_iN,x_{i+1}N,N,\ldots,N) &= \sum_{n_1,\ldots,n_i \in N} f_{k-i+1}([x_1n_1,\ldots,x_in_i], x_{i+1}N,N,\ldots,N) \\
&\leq \sum_{n_1,\ldots,n_i \in N} f_{k-i+1}([x_1n_1,\ldots,x_in_i], N,N,\ldots,N) \\
&= f_k(x_1N,\ldots,x_iN,N,N,\ldots,N),
\end{align*}
as required.
\end{proof}

\subsection{Sketch of the proof of \texorpdfstring{\cref{prop:gallagher}}{Proposition \ref{prop:gallagher}}} \label{ssec:sketch}

Here we give a proof of \cref{prop:gallagher}, omitting proofs of two equalities to be proved in \cref{ssec:details}. Throughout this section, fix $g \in G$ and $xN \in G/N$.

We aim to show that 
\[ f_k(xN,g,N,\ldots,N) \leq f_k(N,g,N,\ldots,N), \]
or in other words,
\[ \sum_{(n_3,\ldots,n_{k+1}) \in N^{k-1}} f_k(xN,g,n_3,\ldots,n_{k+1}) \leq \sum_{(n_3,\ldots,n_{k+1}) \in N^{k-1}} f_k(N,g,n_3,\ldots,n_{k+1}). \]

The idea of the proof is to split this up into smaller parts: that is, to find a partition $N_1 \sqcup \cdots \sqcup N_p$ of $N^{k-1}$ such that
\begin{equation} \label{eqn:frN}
\sum_{(n_3,\ldots,n_{k+1}) \in N_q} f_k(xN,g,n_3,\ldots,n_{k+1}) \leq \sum_{(n_3,\ldots,n_{k+1}) \in N_q} f_k(N,g,n_3,\ldots,n_{k+1})
\end{equation}
for each $q \in \{ 1,\ldots,p \}$. The proof relies on periodic behaviour (in a certain sense) of the numbers $f_k(x^iN,g,n_3,\ldots,n_{k+1})$ where $i \in \Z$ and $(n_3,\ldots,n_{k+1}) \in N_q$. In particular, each part $N_q$ will be subdivided further: in \cref{ssec:defns} we will define a function
\[ L: N_q \to \Z/d\Z \]
for some $d = d(q) \in \N$, with the property that, for any $i \in \Z$ and $(n_3,\ldots,n_{k+1}) \in N_q$, the number $f_k(x^iN,g,n_3,\ldots,n_{k+1})$ depends only on the value of $L(n_3,\ldots,n_{k+1})+i$ in $\Z/d\Z$. That is, given any $(n_3,\ldots,n_{k+1}),(\tilde{n}_3,\ldots,\tilde{n}_{k+1}) \in N_q$ and $i,\tilde{i} \in \Z$, we have
\begin{equation} \label{eqn:period}
\begin{aligned}
\text{if} \quad &L(n_3,\ldots,n_{k+1})+i = L(\tilde{n}_3,\ldots,\tilde{n}_{k+1})+\tilde{i} \quad \text{(in $\Z/d\Z$)}, \\
\text{then} \quad &f_k(x^iN,g,n_3,\ldots,n_{k+1}) = f_k(x^{\tilde{i}}N,g,\tilde{n}_3,\ldots,\tilde{n}_{k+1});
\end{aligned}
\end{equation}
we will prove \eqref{eqn:period} in \cref{ssec:details}.

This implies that there exist some integers $h_j$ (where $j \in \Z/d\Z$) such that
\[ f_k(x^iN,g,n_3,\ldots,n_{k+1}) = h_{L(n_3,\ldots,n_{k+1})+i} \]
for all $i \in \Z$ and $(n_3,\ldots,n_{k+1}) \in N_q$. Thus
\[ \sum_{(n_3,\ldots,n_{k+1}) \in N_q} f_k(N,g,n_3,\ldots,n_{k+1}) = \sum_{j \in \Z/d\Z} |L^{-1}(j)| h_j \]
and
\[ \sum_{(n_3,\ldots,n_{k+1}) \in N_q} f_k(xN,g,n_3,\ldots,n_{k+1}) = \sum_{j \in \Z/d\Z} |L^{-1}(j)| h_{j+1}, \]
and so \eqref{eqn:frN} becomes
\[ \sum_{j \in \Z/d\Z} |L^{-1}(j)| h_{j+1} \leq \sum_{j \in \Z/d\Z} |L^{-1}(j)| h_j. \]
Moreover, in \cref{ssec:details} we will show that
\begin{equation} \label{eqn:adjacent}
|L^{-1}(j)| h_{j+1} = |L^{-1}(j+1)| h_j
\end{equation}
for each $j \in \Z/d\Z$. \cref{prop:gallagher} then follows from the following Lemma:
\begin{lemma}
Let $d \in \mathbb{N}$ and for each $j \in \mathbb{Z}/d\mathbb{Z}$, let $r_j$ and $h_j$ be non-negative integers such that $r_jh_{j+1} = r_{j+1}h_j$ for each $j$. Then
\[ \sum_{j \in \mathbb{Z}/d\mathbb{Z}} r_j h_{j+1} \leq \sum_{j \in \mathbb{Z}/d\mathbb{Z}} r_j h_j. \]
\end{lemma}

\begin{proof}
For a fixed $j \in \mathbb{Z}/d\mathbb{Z}$, since $r_jh_{j+1} = r_{j+1}h_j$, we have either $r_j \leq r_{j+1}$ and $h_j \leq h_{j+1}$, or $r_j \geq r_{j+1}$ and $h_j \geq h_{j+1}$. This implies that
\begin{align*}
0 &\leq \sum_{j \in \mathbb{Z}/d\mathbb{Z}} (r_j-r_{j+1})(h_j-h_{j+1}) \\
&= \sum_{j \in \mathbb{Z}/d\mathbb{Z}} r_jh_j - \sum_{j \in \mathbb{Z}/d\mathbb{Z}} r_jh_{j+1} - \sum_{j \in \mathbb{Z}/d\mathbb{Z}} r_{j+1}h_j + \sum_{j \in \mathbb{Z}/d\mathbb{Z}} r_{j+1}h_{j+1} \\
&= 2 \left( \sum_{j \in \mathbb{Z}/d\mathbb{Z}} r_jh_j - \sum_{j \in \mathbb{Z}/d\mathbb{Z}} r_jh_{j+1} \right),
\end{align*}
as required.
\end{proof}

\subsection{Combinatorial structure of \texorpdfstring{$\mathcal{N}_k(G)$}{Nk(G)}} \label{ssec:defns}

Here we clarify the notation used in \cref{ssec:sketch}. In particular, we define the subsets $N_q \in N^{k-1}$, and for a given $q \in \{1,\ldots,p\}$, the number $d \in \N$ and the function $L: N_q \to \Z/d\Z$ used in \cref{ssec:sketch}.

A key fact used in the argument is the following commutator identity:

\begin{lemma} \label{lem:mainid}
For any $z,y \in G$ and $n_3,\ldots,n_{k+1} \in N$,
\[ [zy,g,n_3,\ldots,n_{k+1}] = \left[z,g,n_3^{\alpha_3^{-1}},\ldots,n_{k+1}^{\alpha_{k+1}^{-1}}\right]^{\alpha_{k+2}} [y,g,n_3,\ldots,n_{k+1}], \]
where $\alpha_i = y \prod_{j=2}^{i-2} [y,g,n_3,\ldots,n_j]$ for $3 \leq i \leq k+2$ (for the avoidance of doubt, $\alpha_3=y$ and $\alpha_4=y[y,g]$).
\end{lemma}
\begin{proof}
We proceed by induction on $k$. We make repeated use of the commutator identity
\begin{equation} \label{eqn:abc}
[ab,c] = [a,c]^b [b,c].
\end{equation}

For the base case $k = 1$, we use \eqref{eqn:abc} with $a = z$, $b = y$ and $c = g$, by noting that $\alpha_3 = y$.

For $k \geq 2$,  the inductive hypothesis gives
\[ [zy,g,n_3,\ldots,n_{k+1}] = \left[ \overbrace{\left[z,g,n_3^{\alpha_3^{-1}},\ldots,n_k^{\alpha_k^{-1}}\right]^{\alpha_{k+1}}}^{a_{k+1}} \overbrace{[y,g,n_3,\ldots,n_k]}^{b_{k+1}}, n_{k+1} \right]. \]
Since $\alpha_{k+2} = \alpha_{k+1} b_{k+1}$, the result follows by applying \eqref{eqn:abc} with $a = a_{k+1}$, $b = b_{k+1}$ and $c = n_{k+1}$.
\end{proof}

This Lemma motivates the following construction. Let $\Gamma$ be a directed labelled multigraph (that is, a directed labelled graph in which loops and multiple edges are allowed) with vertex set
\[ V(\Gamma) = N^{k-1} \]
and edge set
\[ E(\Gamma) = \left\{ \left(n_3,\ldots,n_{k+1}\right) \xrightarrow{y} \left(n_3^{\alpha_3^{-1}},\ldots,n_{k+1}^{\alpha_{k+1}^{-1}}\right) \:\middle|\: y \in xN, n_3,\ldots,n_{k+1} \in N, [y,g,n_3,\ldots,n_{k+1}] = 1 \right\}, \]
where $\alpha_3 = \alpha_3(y), \alpha_4 = \alpha_4(y),\alpha_5 = \alpha_5(y,n_3),\ldots,\alpha_{k+1} = \alpha_{k+1}(y,n_3,\ldots,n_{k-2}) \in G$ are as in \cref{lem:mainid}.

Now write $\Gamma$ as a union of its connected components, 
\[ \Gamma = \Gamma_1 \sqcup \cdots \sqcup \Gamma_p, \]
and write $N_q$ for $V(\Gamma_q)$, where $1 \leq q \leq p$. This defines a partition
\[ N^{k-1} = N_1 \sqcup \cdots \sqcup N_p, \]
as above. Fix $q \in \{1,\ldots,p\}$. In what follows, a walk in $\Gamma_q$ is not required to follow directions of the edges, but does have a choice of orientation associated with it.

\begin{definition} \begin{enumerate}
\item For a walk $\gamma$ of length $s_++s_-$ from $v \in N_q$ to $w \in N_q$, define the \emph{directed length} $\ell(\gamma)$ of $\gamma$ to be $s_+-s_-$, where $s_+$ (respectively $s_-$) is the number of edges in $\gamma$ with directions coincident with (respectively opposite to) the direction of $\gamma$. Note that given any walk $\gamma$ in $\Gamma_q$ we have $\ell(\gamma^{-1})=-\ell(\gamma)$.

\item Define the \emph{period} of $\Gamma_q$ to be
\[ d = \gcd (\{o\} \cup \{ |\ell(c)| \mid c \text{ is a closed walk in } \Gamma_q \}), \]
where $o$ is the order of $xN$ in $G/N$.
\item Choose a base vertex $v_q$ of $\Gamma_q$. For any vertex $v \in \Gamma_q$, define the \emph{level} of $v$ to be 
\[ L(v) = \ell(\gamma_v)+d\mathbb{Z} \in \mathbb{Z}/d\mathbb{Z} \]
where $\gamma_v$ is a walk in $\Gamma_q$ from $v_q$ to $v$. Note that if $\gamma_v$, $\tilde\gamma_v$ are two such walks, then $c = \gamma_v^{-1} \tilde\gamma_v$ is a \emph{closed} walk, and so $d$ divides $\ell(c) = -\ell(\gamma_v) + \ell(\tilde\gamma_v)$ by the choice of $d$. Thus $L(v)$ does not depend on the choice of $\gamma_v$.
\end{enumerate}
\end{definition}
\begin{remark*}
The set $\mathcal{W}(\Gamma_q)$ of walks in $\Gamma_q$ forms a group under concatenation, with inverses given by changing orientation, and in this setting $\ell: \mathcal{W}(\Gamma_q) \to \Z$ is a homomorphism.
\end{remark*}

\subsection{Completing the proof of \texorpdfstring{\cref{prop:gallagher}}{Proposition \ref{prop:gallagher}}} \label{ssec:details}

We now prove \eqref{eqn:period} and \eqref{eqn:adjacent} from \cref{ssec:sketch}, which will complete the proof of \cref{prop:gallagher}.

The last part of the following Lemma shows \eqref{eqn:period} is true:
\begin{lemma} \begin{enumerate}
\item \label{item:lperiod1} For any walk $\gamma$ from $(n_3,\ldots,n_{k+1}) \in N_q$ to $(\tilde{n}_3,\ldots,\tilde{n}_{k+1}) \in N_q$ and any $i \in \Z$, we have
\[ f_k(x^iN,g,n_3,\ldots,n_{k+1}) = f_k(x^{i-\ell(\gamma)}N,g,\tilde{n}_3,\ldots,\tilde{n}_{k+1}). \]
\item \label{item:lperiod2} For any $(n_3,\ldots,n_{k+1}) \in N_q$ and $i \in \Z$, we have
\[ f_k(x^iN,g,n_3,\ldots,n_{k+1}) = f_k(x^{i-d}N,g,n_3,\ldots,n_{k+1}). \]
\item \label{item:lperiod3} For any $(n_3,\ldots,n_{k+1}),(\tilde{n}_3,\ldots,\tilde{n}_{k+1}) \in N_q$ and $i,\tilde{i} \in \Z$, if
\[ L(n_3,\ldots,n_{k+1})+i = L(\tilde{n}_3,\ldots,\tilde{n}_{k+1})+\tilde{i} \quad \text{(in $\Z/d\Z$)}, \]
then
\[ f_k(x^iN,g,n_3,\ldots,n_{k+1}) = f_k(x^{\tilde{i}}N,g,\tilde{n}_3,\ldots,\tilde{n}_{k+1}). \]
\end{enumerate}
\end{lemma}

\begin{proof} \begin{enumerate}

\item We proceed by induction on the length of $\gamma$. For the base case (when $\gamma$ is an edge), note that by \cref{lem:mainid}, an edge $(n_3,\ldots,n_{k+1}) \xrightarrow{y} (\tilde{n}_3,\ldots,\tilde{n}_{k+1})$ in $\Gamma_q$ defines a bijection between elements $zy \in x^iN$ with $[zy,g,n_3,\ldots,n_{k+1}] = 1$ and elements $z \in x^{i-1}N$ with $[z,g,\tilde{n}_3,\ldots,\tilde{n}_{k+1}] = 1$, and hence
\begin{equation} \label{eqn:gequal}
f_k((xN)^i,g,n_3,\ldots,n_{k+1}) = f_k((xN)^{i-1},g,\tilde{n}_3,\ldots,\tilde{n}_{k+1}).
\end{equation}
For the inductive step (when the length of $\gamma$ is at least $2$), note that we can write $\gamma = \tilde\gamma e^{\varepsilon}$ for some $e \in E(\Gamma_q)$, $\varepsilon \in \{ \pm 1 \}$, and a walk $\tilde\gamma$ that is strictly shorter than $\gamma$. Thus, applying the inductive hypothesis to $\tilde\gamma$ and \eqref{eqn:gequal} to $e$ yields the result.

\item Fix a vertex $v = (n_3,\ldots,n_{k+1}) \in N_q$ and $i \in \Z$. By definition of $d$, there exist closed walks $c_1,\ldots,c_r$ and integers $m,m_1,\ldots,m_r \in \Z$ such that
\[ d = mo + m_1 \ell(c_1) + \cdots + m_r \ell(c_r). \]
Note that we may transform the closed walks $c_j$ to ones that start and end at $v$: indeed, if $\gamma_j$ is a walk from $v$ to the starting (and ending) vertex of $c_j$, then $\tilde{c}_j = \gamma_j c_j \gamma_j^{-1}$ is a closed walk starting and ending at $v$, and $\ell(\tilde{c}_j) = \ell(c_j)$. This allows us to construct a closed walk
\[ \tilde{c} = \tilde{c}_1^{m_1} \cdots \tilde{c}_r^{m_r} \]
and we have
\[ \ell(\tilde{c}) = m_1\ell(\tilde{c}_1) + \cdots + m_r\ell(\tilde{c}_r) = d-mo. \]
Substituting $\gamma = \tilde{c}$ to part \eqref{item:lperiod1} yields
\[ f(x^iN,g,n_3,\ldots,n_{k+1}) = f_k(x^{i-d+mo}N,g,n_3,\ldots,n_{k+1}). \]
But since $o$ is the order of $xN$ in $G/N$, we get $x^{i-d+mo}N = (x^{i-d}N)((xN)^o)^m = x^{i-d}N$, which gives the result.

\item Let $\gamma$ (respectively $\tilde\gamma$) be a walk in $\Gamma_q$ from the base vertex $v_q$ to $(n_3,\ldots,n_{k+1})$ (respectively $(\tilde{n}_3,\ldots,\tilde{n}_{k+1})$). By definition of level, we have
\[ \ell(\gamma^{-1}\tilde\gamma) + d\Z = -\ell(\gamma)+\ell(\tilde\gamma) +d\Z = -L(n_3,\ldots,n_{k+1}) + L(\tilde{n}_3,\ldots,\tilde{n}_{k+1}) = i-\tilde{i} + d\Z, \]
and so $\ell(\gamma^{-1}\tilde\gamma) = i-\tilde{i}+md$ for some $m \in \Z$. By part \eqref{item:lperiod1}, we have
\[ f_k(x^iN,g,n_3,\ldots,n_{k+1}) = f_k(x^{i-(i-\tilde{i}+md)}N,g,\tilde{n}_3,\ldots,\tilde{n}_{k+1}) = f_k(x^{\tilde{i}-md}N,g,\tilde{n}_3,\ldots,\tilde{n}_{k+1}), \]
and so $|m|$ applications of part \eqref{item:lperiod2} to the right hand side gives the result.
\qedhere

\end{enumerate}
\end{proof}

Finally, we prove \eqref{eqn:adjacent}:

\begin{lemma}
For each $j \in \Z/d\Z$, we have $|L^{-1}(j)|h_{j+1} = |L^{-1}(j+1)|h_j$.
\end{lemma}

\begin{proof}
We will give a bijection between the set
\[ \mathcal{A} = \bigsqcup_{\substack{(n_3,\ldots,n_{k+1}) \in N_q \\ L(n_3,\ldots,n_{k+1}) = j}} \mathcal{N}_k(G) \cap \left( xN \times \{g\} \times \{n_3\} \times \ldots \times \{n_{k+1}\} \right) \]
and the set
\[ \mathcal{B} = \bigsqcup_{\substack{(\tilde{n}_3,\ldots,\tilde{n}_{k+1}) \in N_q \\ L(\tilde{n}_3,\ldots,\tilde{n}_{k+1}) = j+1}} \mathcal{N}_k(G) \cap \left( x^{-1}N \times \{g\} \times \{\tilde{n}_3\} \times \ldots \times \{\tilde{n}_{k+1}\} \right). \]
Since $\mathcal{A}$ is a disjoint union of $|L^{-1}(j)|$ sets, each of cardinality $h_{j+1}$, and $\mathcal{B}$ is a disjoint union of $|L^{-1}(j+1)|$ sets, each of cardinality $h_j$, this will imply the result.

Now consider
\begin{align*}
\theta: xN \times \{g\} \times N^{k-1} &\to x^{-1}N \times \{g\} \times N^{k-1}, \\
(y,g,n_3,\ldots,n_{k+1}) &\mapsto \left(y^{-1},g,n_3^{\alpha_3^{-1}},\ldots,n_{k+1}^{\alpha_{k+1}^{-1}}\right),
\end{align*}
where $\alpha_3 = \alpha_3(y)$, $\alpha_4 = \alpha_4(y)$, $\alpha_5 = \alpha_5(y,n_3)$, \ldots, $\alpha_{k+1} = \alpha_{k+1}(y,n_3,\ldots,n_{k-1})$ are as in \cref{lem:mainid}. First, we claim that $\theta$ is a bijection. Indeed, for each $i$, the element $\alpha_i$ does not depend on $n_i,\ldots,n_{k+1}$, and so it follows (by induction on $k-i$) that the restriction of $\theta$ given by
\[ \theta_i : \{y\} \times \{g\} \times \{n_3\} \times \cdots \times \{n_{i+1}\} \times N^{k-i} \to \left\{y^{-1}\right\} \times \left\{g\right\} \times \left\{n_3^{\alpha_3^{-1}}\right\} \times \cdots \times \left\{n_{i+1}^{\alpha_{i+1}^{-1}}\right\} \times N^{k-i}\]
is a bijection for each $y \in xN$ and $(n_3,\ldots,n_{i+1}) \in N^{i-1}$. In particular,
\[ \theta_1: \{y\} \times \{g\} \times N^{k-1} \to \left\{y^{-1}\right\} \times \{g\} \times N^{k-1} \]
is a bijection for each $y \in xN$, and hence $\theta$ is a bijection as well.

It is now enough to show that $\theta(\mathcal{A}) = \mathcal{B}$. By substituting $z = y^{-1}$ in \cref{lem:mainid}, it follows that $[y,g,n_3,\ldots,n_{k+1}] = 1$ if and only if $\left[y^{-1},g,n_3^{\alpha_3^{-1}},\ldots,n_{k+1}^{\alpha_{k+1}^{-1}}\right] = 1$, and hence that
\[ \theta \left( \mathcal{N}_k(G) \cap (xN \times \{g\} \times N^{k-1}) \right) = \mathcal{N}_k(G) \cap (x^{-1}N \times \{g\} \times N^{k-1}). \]
Furthermore, for an arbitrary edge $\left(n_3,\ldots,n_{k+1}\right) \xrightarrow{y} \left(\tilde{n}_3,\ldots,\tilde{n}_{k+1}\right)$ in $\Gamma$ (note that $\left(\tilde{n}_3,\ldots,\tilde{n}_{k+1}\right) = \left(n_3^{\alpha_3^{-1}},\ldots,n_{k+1}^{\alpha_{k+1}^{-1}}\right)$ in this case), its endpoints are in the same connected component of $\Gamma$, that is, $\left(n_3,\ldots,n_{k+1}\right) \in N_q$ if and only if $\left(\tilde{n}_3,\ldots,\tilde{n}_{k+1}\right) \in N_q$. Moreover, if it is the case that $\left(n_3,\ldots,n_{k+1}\right),\left(\tilde{n}_3,\ldots,\tilde{n}_{k+1}\right) \in N_q$ then by definition of level we have
\[
L\left(\tilde{n}_3,\ldots,\tilde{n}_{k+1}\right)=L\left(n_3,\ldots,n_{k+1}\right)+1.
\]
Hence $\theta(\mathcal{A}) = \mathcal{B}$, as required.
\end{proof}

\subsection{Residually finite groups} \label{ssec:res.fin}

Finally, we prove \cref{thm:res.fin.}. We follow the argument of Antol\'in and the first and fourth authors in \cite[Theorem 1.3]{amv}. The proof uses the following result.

\begin{theorem}[Erfanian, Rezaei, Lescot {\cite[Theorem 5.1]{erl}}] \label{thm:gap}
Let $G$ be a finite group that is not $k$-step nilpotent. Then
\[
\dc^k(G) \leq \frac{2^{k+2}-3}{2^{k+2}}.
\]
\end{theorem}

\begin{proof}[Proof of \cref{thm:res.fin.}]
We describe a recursive process outputting a (possibly finite) sequence $G_0>G_1>G_2>\ldots$ of members of $\mathcal N$ as follows. We may assume without loss of generality that $G\in\mathcal N$ and set $G_0=G$. Once $G_{i-1}$ is defined, if it is $k$-step nilpotent we terminate the process. If not, there exist $x_1,\ldots,x_{k+1}\in G_{i-1}$ such that $[x_1,\ldots,x_{k+1}]\ne1$. Since $\bigcap_{N\in\mathcal N}N=\{1\}$, there therefore exists $N_i\in\mathcal N$ such that $[x_1,\ldots,x_{k+1}]\notin N_i$. Set $G_i=G_{i-1}\cap N_i$, noting that $G_i\in\mathcal N$ by the finite-intersection property, and that $G_{i-1}/G_i$ is not $k$-step nilpotent.

Writing $\gamma_k=(2^{k+2}-3)/2^{k+2}$, it follows from \cref{thm:gap} that $\dc^k(G_{i-1}/G_i)\le\gamma_k$ for every $i$, and hence from \cref{thm:gallagher} that for every $n$ with $G_n$ defined we have
\[
\dc^k(G/G_n)\le\prod_{i=1}^n\dc^k(G_{i-1}/G_i)\le\gamma_k^n.
\]
The process must therefore terminate for some $n\le\log\alpha/\log\gamma_k$, meaning that $G_n$ is a $k$-step nilpotent subgroup of finite index in $G$.
\end{proof}

\section{Dependence on rank}
\label{sec:rank}
Here we give an example, for any odd prime $p$ and any $k \in \N$, of a family $\left(G^{(n)}\right)_{n=1}^\infty$ of finite $p$-groups that are $(k+1)$-step nilpotent but not $k$-step nilpotent, and such that the centre $Z\left(G^{(n)}\right)$ of $G^{(n)}$ has order $p$. Moreover, we will show that any $k$-step nilpotent subgroup $K^{(n)}$ of $G^{(n)}$ has index at least $p^n$. As $G^{(n)}/Z\left(G^{(n)}\right)$ is $k$-step nilpotent, this will show that the bound on the index of a $k$-step nilpotent subgroup of $G$ in \cref{cor:fbn} has to depend on the rank of $G$. By \cref{prop:converse}, the same can be said about the bound in \cref{thm:dc_k}.

Furthermore, note that this example will show that the index of a $k$-step nilpotent subgroup in \cref{thm:res.fin.} cannot be bounded in terms of $k$ and $\alpha$. To see this, it is enough to apply \cref{prop:converse} and to note that if $\dc^k\left(G^{(n)}\right) \geq \alpha$ then also $\dc^k\left(G^{(n)}/N\right) \geq \alpha$ for any normal subgroup $N \lhd G^{(n)}$.

Throughout this section, we fix an odd prime $p$, and denote the finite field of cardinality $p$ by $\Fp$. For $r,s \in \N$, we denote by $\mat{r}{s}$ the $\Fp$-vector space of $r \times s$ matrices with entries in $\Fp$.

\subsection{The group \texorpdfstring{$G_k(n,r,s)$}{Gk(n,r,s)}}
\label{ssec:Gknrs}

Let $k \in \Z_{\geq 0}$ and let $n,r,s \in \N$. We consider the following subgroup of $GL_{r+kn+s}(\Fp)$ consisting of block upper unitriangular matrices:
\[
G_k(n,r,s) = \left\{ \begin{pmatrix}
I_r & A_0 & A_1 & \cdots & A_{k-1} & C \\
& I_n & D_{1,1} & \cdots & D_{1,k-1} & B_1 \\
&& I_n & \ddots & \vdots & \vdots \\
&&& \ddots & D_{k-1,k-1} & B_{k-1} \\
&&&& I_n & B_k \\
&&&&& I_s
\end{pmatrix} \:\middle|\: \begin{array}{@{}l@{}} A_i \in \mat{r}{n} \\ \qquad \text{for } 0 \leq i \leq k-1, \\ B_i \in \mat{n}{s} \\ \qquad \text{for } 1 \leq i \leq k, \\ C \in \mat{r}{s}, \\ D_{i,j} \in \mat{n}{n} \\ \qquad \text{for } 1 \leq i \leq j \leq k-1 \end{array} \right\}.
\]
For a matrix $X \in G_k(n,r,s)$, we will write $A_j(X)$, $B_i(X)$, $C(X)$ and $D_{i,j}(X)$ for the corresponding blocks of $X$. For a subset $U \subseteq G_k(n,r,s)$ we will similarly write $A_j(U) = \{ A_j(X) \mid X \in U \}$, etc.

Note that for $k = 0$, the group $G_0(n,r,s) = \left\{ \begin{pmatrix} I_r & C \\ 0 & I_s \end{pmatrix} \:\middle|\: C \in \mat{r}{s} \right\}$ is just the elementary abelian group of order $p^{rs}$. For $k = r = s = 1$, the group $G_1(n,1,1)$ is the extraspecial group of exponent $p$. It is well-known that such a group is $2$-step nilpotent, has centre of order $p$, but no abelian subgroups of index $< p^n$ (see, for instance, \cref{lem:Gcentralseries} and \cite[Theorem 1.8]{verardi}). We aim to generalise this example; in particular, for the sequence $\left(G^{(n)}\right)$ of groups described above we will take $G^{(n)} = G_k(n,1,1)$. We thus need to show that $G_k(n,1,1)$ is $(k+1)$-step nilpotent, has centre of order $p$ and has no $k$-step nilpotent subgroups of index $< p^n$.

The first two of these statements follow from the following Lemma, whose proof is easy and left as an exercise for the reader.

\begin{lemma} \label{lem:Gcentralseries}
Let $k\in\Z_{\geq 0}$ and $n,r,s \in \N$. Let $G = G_k(n,r,s)$, and let $G = \gamma_1(G) \geq \gamma_2(G) \geq \cdots$ and $\{1\} = Z_0(G) \leq Z_1(G) \leq \cdots$ be the lower and upper central series of $G$, respectively. Then
\begin{align*}
\gamma_{\ell+1}(G) = Z_{k+1-\ell}(G) = \{ X \in G \mid\ &A_j(X) = 0 \text{ for } j < \ell,\ B_i(X) = 0 \text{ for } k-i < \ell, \\ &D_{i,j}(X) = 0 \text{ for } j-i < \ell \}
\end{align*}
for all $\ell \in \{ 0,\ldots,k \}$. \qed
\end{lemma}

We are therefore left to show that $G_k(n,1,1)$ has no $k$-step nilpotent subgroups of index $< p^n$. In \cref{ssec:no-large-sgps} we will prove the following proposition, which is slightly more general.

\begin{prop} \label{prop:no-large-sgps}
Let $k,n,r,s \in \N$. If a subgroup $K \leq G_k(n,r,s)$ has index $< p^n$, then $K$ is not $k$-step nilpotent.
\end{prop}

\begin{remark}
Note that in the case $r = s = 1$, the bound in \cref{prop:no-large-sgps} is sharp: indeed, $\{ X \in G_k(n,1,1) \mid A_0(X) = 0 \}$ is a subgroup of $G_k(n,1,1)$ of index $p^n$, and it is not hard to verify that it is $k$-step nilpotent.
\end{remark}

\subsection{Non-existence of large \texorpdfstring{$k$}{k}-step nilpotent subgroups}
\label{ssec:no-large-sgps}

Let $G = G_k(n,r,s)$. By \cref{lem:Gcentralseries}, the abelianisation map $\rho: G \to G^{ab}$ is given by mapping a matrix in $G$ to the set of its superdiagonal blocks:
\begin{align*}
\rho: G &\to \mat{r}{n} \oplus \left( \bigoplus_{i=1}^{k-1} \mat{n}{n} \right) \oplus \mat{n}{s} \cong \Fp^{n(r+(k-1)n+s)}, \\
X &\mapsto (A_0(X),D_{1,1}(X),\ldots,D_{k-1,k-1}(X),B_k(X))
\end{align*}
for $k \geq 1$, and $\rho: G \to \mat{r}{s}, X \mapsto C(X)$ for $k = 0$.
For a subgroup $K \leq G$, we define the \emph{quasi-rank} (respectively \emph{quasi-corank}) of $K$ in $G$ to be the dimension (respectively codimension) of $\rho(K)$ in the $\F_p$-vector space $G^{ab}$. Note that if $K$ has quasi-corank $q$ then we have $[G:K\gamma_2(G)] = p^q$. We thus aim to show that the quasi-corank of a $k$-step nilpotent subgroup of $G_k(n,r,s)$ will be at least $n$.

The inductive proof of \cref{prop:no-large-sgps} is based on the surjective homomorphism $\pi = \pi_{k,n,r,s}$, obtained by taking the bottom-right $(kn+s) \times (kn+s)$ submatrix:
\begin{align*}
\pi: G_k(n,r,s) &\to G_{k-1}(n,n,s), \\
 \begin{pmatrix}
I_r & A_0 & A_1 & \cdots & A_{k-1} & C \\
& I_n & D_{1,1} & \cdots & D_{1,k-1} & B_1 \\
&& I_n & \ddots & \vdots & \vdots \\
&&& \ddots & D_{k-1,k-1} & B_{k-1} \\
&&&& I_n & B_k \\
&&&&& I_s
\end{pmatrix} &\mapsto  \begin{pmatrix}
I_n & D_{1,1} & \cdots & D_{1,k-1} & B_1 \\
& I_n & \ddots & \vdots & \vdots \\
&& \ddots & D_{k-1,k-1} & B_{k-1} \\
&&& I_n & B_k \\
&&&& I_s
\end{pmatrix}.
\end{align*}
Note that if $K \leq G_k(n,r,s)$ has quasi-corank $q$, then $\pi(K) \leq G_{k-1}(n,n,s)$ will have quasi-corank at most $q$.

For any $X \in \gamma_k(G_k(n,r,s))$ we have $B_2(X) = \cdots = B_k(X) = 0$ by \cref{lem:Gcentralseries}, and for any $Y \in \ker\pi_{k,n,r,s}$ we have $B_1(Y) = \cdots = B_k(Y) = 0$ by the definition of $\pi_{k,n,r,s}$. Therefore,
\begin{equation} \label{eqn:CXY}
\begin{aligned}
C([X,Y]) &= \left( C(Y)+A_0(X)B_1(Y)+\cdots+A_{k-1}(X)B_k(Y)+C(X) \right)\\ &-\left( C(X)+A_0(Y)B_1(X)+\cdots+A_{k-1}(Y)B_k(X)+C(Y) \right) \\ &= -A_0(Y)B_1(X) \qquad \text{for all } X \in \gamma_k(G_k(n,r,s)) \text{ and } Y \in \ker\pi_{k,n,r,s}.
\end{aligned}
\end{equation}
Thus, in order to prove \cref{prop:no-large-sgps}, given a subgroup $K \leq G_k(n,r,s)$ of quasi-corank $< n$ we need to find matrices $X \in \gamma_k(K)$ and $Y \in K\cap\ker\pi$ such that $A_0(Y)B_1(X) \neq 0$.

We first prove a slightly stronger version of \cref{prop:no-large-sgps} under the additional assumption that $r = n$.

\begin{lemma} \label{lem:qcorank}
Let $k \in \Z_{\geq 0}$ and $n,s \in \N$. Let $K$ be a subgroup of $G_k(n,n,s)$ of quasi-corank $q < n$. Then the subspace
\[
C(\gamma_{k+1}(K)) = \{ C(X) \mid X \in \gamma_{k+1}(K) \} \leq \mat{n}{s}
\]
has codimension at most $q$.
\end{lemma}

\begin{proof}
By induction on $k$. For $k = 0$, we have $G_0(n,n,s) \cong \mat{n}{s}$ and $C(\gamma_1(K)) = C(K) \cong K$, hence the result is clear.

Now suppose $k \geq 1$, and let $\pi = \pi_{k,n,n,s}$. As $K$ has quasi-corank $q$ in $G_k(n,n,s)$, the subgroup $\pi(K) \leq G_{k-1}(n,n,s)$ will have quasi-corank at most $q$. Therefore, by induction hypothesis, the subspace
\[
C(\gamma_k(\pi(K))) = B_1(\gamma_k(K)) = \{ B_1(X) \mid X \in \gamma_k(K) \}  \leq \mat{n}{s}
\]
will have codimension at most $q$.

Moreover, it is clear by the definition of the quasi-corank that the subspace
\[
A_0(K\cap\ker\pi) := \{ A_0(X) \mid X \in K\cap\ker\pi \} \leq \mat{n}{n}
\]
will have codimension at most $q$, so in particular
\[
\dim A_0(K\cap\ker\pi) \geq n^2-q > n^2-n.
\]
It follows by \cite[Corollary 13]{dsp} that $A_0(K\cap\ker\pi)$ is generated by matrices of rank $n$, so in particular there exists a matrix $Y \in K\cap\ker\pi$ such that $A_0(Y)$ is invertible. But now, as $C(\gamma_{k+1}(K))$ contains $C([X,Y]) = -A_0(Y)B_1(X)$ for any $X \in \gamma_k(K)$ (see \eqref{eqn:CXY}), it follows that
\[
\codim C(\gamma_{k+1}(K)) \leq \codim B_1(\gamma_k(K)) \leq q,
\]
as required.
\end{proof}

\begin{proof}[Proof of \cref{prop:no-large-sgps}]
Let $q$ be the quasi-corank of $K$ in $G = G_k(n,r,s)$. Then we have
\[
p^q = [G:K\gamma_2(G)] \leq [G:K] < p^n
\]
and so $q < n$. Consider again the map $\pi = \pi_{k,n,r,s}$, and let $q_1$ be the quasi-corank of $\pi(K)$ in $G_{k-1}(n,n,s)$. By \cref{lem:qcorank}, the subspace $B_1(\gamma_k(K)) = C(\gamma_k(\pi(K)))$ will have codimension at most $q_1$ in $\mat{n}{s}$. By the rank-nullity theorem, the subspace $A_0(K\cap\ker\pi) \leq \mat{r}{n}$ will have codimension $q-q_1 =: q_2$.

Now consider the projections $\tau_1: \mat{r}{n} \to \Fp^n$ and $\tau_2: \mat{n}{s} \to \Fp^n$ of matrices to the top row and to the right column, respectively. By \eqref{eqn:CXY}, for any $X \in \gamma_k(K)$ and $Y \in K\cap\ker\pi$, the top right entry of $[X,Y]$ will be $-\langle \tau_1(A_0(Y)), \tau_2(B_1(X)) \rangle$, where $\langle -,- \rangle$ is the standard bilinear form on $\Fp^n$. Furthermore, it is clear that $T_1 := \tau_1(A_0(K\cap\ker\pi))$ and $T_2 := \tau_2(B_1(\gamma_k(K)))$ will have codimensions (in $\Fp^n$) at most $q_1$ and at most $q_2$, respectively. Thus, as $q < n$, we have
\[
\dim T_1 + \dim T_2 \geq (n-q_1) + (n-q_2) = 2n-q > n,
\]
and so, as $\langle -,- \rangle$ is non-degenerate,
\[
\dim T_1 > n-\dim T_2 = \dim T_2^\perp.
\]
This implies that $T_1 \nleq T_2^\perp$, that is, $\langle T_1,T_2 \rangle \neq 0$. Therefore, there exist matrices $X \in \gamma_k(K)$ and $Y \in K\cap\ker\pi$ such that the top right entry of $[X,Y]$ is non-zero, so $K$ is not $k$-step nilpotent.
\end{proof}

\subsection{Qualitative conclusions}
\label{ssec:dirlim}

Apart from the rank-dependence of quantitative conclusions of \cref{thm:dc_k,cor:fbn}, we may use the groups $G_k(n,1,1)$ to give counterexamples to qualitative conclusions as well for groups that are not finitely generated. In particular, we will prove the following result.

\begin{prop} \label{prop:needfg}
For any $k \geq 1$ and any odd prime $p$, there exists a group $G$ and a finite normal subgroup $H \lhd G$ of order $p$ such that $G/H$ is $k$-step nilpotent, but $G$ is not virtually $k$-step nilpotent.
\end{prop}

Throughout the rest of this section, fix an odd prime $p$ and, for each $n \geq 1$, let $G_k(n) := G_k(n,1,1)$ be the finite groups defined in \cref{ssec:Gknrs}.

\begin{proof}[Proof of \cref{prop:needfg}]
Our proof relies on the observation that $G_k(n)$ can be seen as a subgroup of $G_k(n+1)$. In particular, it is easy to see that
\[
\overline{G_k(n)} = \left\{ \begin{pmatrix}
1 & \overline{\mathbf{a}_0}^T & \overline{\mathbf{a}_1}^T & \cdots & \overline{\mathbf{a}_{k-1}}^T & c \\
& I_{n+1} & \overline{D_{1,1}} & \cdots & \overline{D_{1,k-1}} & \overline{\mathbf{b}_1} \\
&& I_{n+1} & \ddots & \vdots & \vdots \\
&&& \ddots & \overline{D_{k-1,k-1}} & \overline{\mathbf{b}_{k-1}} \\
&&&& I_{n+1} & \overline{\mathbf{b}_k} \\
&&&&& 1
\end{pmatrix} \:\middle|\: \begin{array}{@{}l@{}} \mathbf{a}_i,\mathbf{b}_i \in \mathbb{F}_p^n \\ \qquad \text{for } 0 \leq i \leq k-1, \\ c \in \mathbb{F}_p, \\ D_{i,j} \in \mat{n}{n} \\ \qquad \text{for } 1 \leq i \leq j \leq k-1 \end{array} \right\}
\]
is a subgroup of $G_k(n+1)$ isomorphic to $G_k(n)$, where given any $A \in \mat{n}{n}$ and $\mathbf{a} \in \mathbb{F}_p^n$ we define $\overline{A} = \begin{pmatrix} A & \mathbf{0} \\ \mathbf{0}^T & 0 \end{pmatrix} \in \mat{(n+1)}{(n+1)}$ and $\overline{\mathbf{a}} = \begin{pmatrix} \mathbf{a} \\ 0 \end{pmatrix} \in \mathbb{F}_p^{n+1}$. This allows us to define the direct limit
\[
G_k := \varinjlim G_k(n).
\]

Given $n \geq 1$, let $f_n: G_k(n) \to G_k$ be the canonical inclusion. It follows from \cref{lem:Gcentralseries} that for each $n$, the image $f_n(Z(G_k(n)))$ of the centre of $G_k(n)$ in $G_k$ is the same subgroup ($H_k$, say) of $G_k$ of order $p$. Hence we have $Z(G_k) = \bigcup_{n \geq 1} f_n(Z(G_k(n))) = H_k$, and in particular, $H_k$ is normal in $G_k$. We will show that $G = G_k$ and $H = H_k$ satisfy the conclusion of the Proposition.

To show that $G_k/H_k$ is $k$-step nilpotent, let $g_0,\ldots,g_k \in G_k$ be arbitrary elements. Then, for any sufficiently large $n$ and all $i \in \{0,\ldots,k\}$ we have $g_i = f_n(h_i)$ for some $h_i \in G_k(n)$, and so
\[
[h_0,\ldots,h_k] \in \gamma_{k+1}(G_k(n)) \leq Z(G_k(n))
\]
as $G_k(n)$ is $(k+1)$-step nilpotent. In particular,
\[
[g_0,\ldots,g_k] = f_n([h_0,\ldots,h_k]) \in f_n(Z(G_k(n))) = H_k,
\]
and so $G_k/H_k$ is $k$-step nilpotent, as required.

Finally, to show that $G_k$ is not virtually $k$-step nilpotent, let $N \leq G_k$ be a subgroup of index $m < \infty$. Let $n \in \Z_{\geq 1}$ be such that $p^n > m$, and note that $[G_k(n) : f_n^{-1}(N)] \leq [G_k : N] = m < p^n$. Thus, by \cref{prop:no-large-sgps}, $f_n^{-1}(N)$ cannot be $k$-step nilpotent. But as $f_n$ is injective, $f_n^{-1}(N)$ is isomorphic to a subgroup of $N$, and so $N$ cannot be $k$-step nilpotent either.
\end{proof}

As the group $G_k$ constructed in Proposition \ref{prop:needfg} is a direct limit of finite groups, $\varinjlim G_k(n)$, it is amenable, and in particular the finite subgroups $G_k(n)$ form a F{\o}lner sequence for $G_k$. We may thus define measures $\mu_n$ on $G_k$ by setting
\[
\mu_n(A) = \frac{|A \cap f_n(G_k(n))|}{|G_k(n)|}
\]
for any $A \subseteq G_k$, where $f_n: G_k(n) \to G_k$ is the canonical inclusion. It follows from a result of the second author \cite[Theorem 1.12]{comm.prob} that the sequence $M = (\mu_n)_{n=1}^\infty$ measures index uniformly on $G$. Moreover, we know that $|\gamma_{k+1}(G_k(n))| = p$ which, when combined with \cref{prop:non-id}, implies that $\dc^k(G_k(n)) \geq 1/p$ and therefore $\dc_M^k(G_k) = \limsup_{n \to \infty} \dc^k(G_k(n)) \geq 1/p$. This shows that the assumption for $G$ to be finitely generated is necessary in \cref{thm:dc_k} as well.

\begin{remark} \label{rem:rfnilp}
Note that the group $G$ in \cref{prop:needfg} cannot be residually finite. Indeed, any finite-index subgroup $N \leq G$ cannot be $k$-step nilpotent, and therefore
\begin{equation} \label{e:HkNnot1}
\{1\} \neq \gamma_{k+1}(N) \leq \gamma_{k+1}(G) \cap N \leq H \cap N,
\end{equation}
where the last inclusion comes from the fact that $G/H$ is $k$-step nilpotent. Hence, $H \cap N \neq \{1\}$; since $H$ has prime order, this implies that $H \subseteq N$.
Thus $H$ is contained in every finite index subgroup of $G$ and so $G$ cannot be residually finite, as claimed.
\end{remark}

\appendix
\section{Polynomial mappings into torsion-free nilpotent groups} \label{sec:polymapnilp}
In this appendix we prove the following extension of \cref{lem:tf.quot}, using a similar argument to the one that Leibman uses to reduce \cite[Proposition 3.21]{leibman} to \cite[Proposition 2.15]{leibman}.
\begin{prop}\label{prop:tf.quot}
Let $G$ be a group, let $N$ be a finitely generated torsion-free $s$-step nilpotent group, and let $\varphi:G\to N$ be a polynomial mapping of degree $d$. Then there is a torsion-free $ds$-step nilpotent quotient $G'$ of $G$ and a polynomial mapping $\hat\varphi:G'\to N$ of degree $d$ such that, writing $\pi:G\to G'$ for the quotient homomorphism, we have $\varphi=\hat\varphi\circ\pi$.
\end{prop}
Given a group $G$ we write
\[
G=\gamma_1(G)\rhd \gamma_2(G)\rhd\ldots
\]
for the lower central series of $G$. Following \cite{mpty}, we define the \emph{generalised commutator subgroups} $\overline{\gamma_i(G)}$ of $G$ via
\[
\overline{\gamma_i(G)}=\{x\in G: \exists n\in\N \text{ such that } x^n\in \gamma_i(G) \},
\]
noting that $G/\overline{\gamma_i(G)}$ is torsion-free $(i-1)$-step nilpotent.
\begin{lemma}\label{lem:comm.fg}
Let $G$ be a group and let $x\in \gamma_i(G)$. Then there exists a finitely generated subgroup $\Gamma=\Gamma(x,i)<G$ such that $x\in\gamma_i(\Gamma)$. If instead $x\in\overline{\gamma_i(G)}$ then there exists a finitely generated subgroup $\Lambda=\Lambda(x,i)<G$ such that $x\in\overline{\gamma_i(\Lambda)}$
\end{lemma}
\begin{proof}
To start with we assume that $x\in \gamma_i(G)$. In the case $i=1$ the lemma is satisfied by taking $\Gamma(x,1)=\langle x\rangle$, so we may assume that $i\ge2$. If $x\in \gamma_i(G)$ this implies that there exist elements $y_1,\ldots,y_k\in \gamma_{i-1}(G)$ and $z_1,\ldots,z_k\in G$ such that $x=\prod_{j=1}^k[y_j,z_j]$, and so by induction on $i$ we may take
\[
\Gamma(x,i)=\langle\Gamma(y_1,i-1),\ldots,\Gamma(y_k,i-1),z_1,\ldots,z_k\rangle.
\]
If instead $x\in\overline{\gamma_i(G)}$ then by definition there exists $n\in\N$ such that $x^n\in \gamma_i(G)$, and so we may take $\Lambda(x,i)=\langle\Gamma(x^n,i),x\rangle$.
\end{proof}

\begin{proof}[Proof of \cref{prop:tf.quot}]
It is sufficient to show that for every $x\in G$ and $c\in\overline{\gamma_{ds+1}(G)}$ we have $\varphi(xc)=\varphi(x)$. Following Leibman's proof of \cite[Proposition 3.21]{leibman}, we may assume by \cref{lem:comm.fg} that $G$ is finitely generated. It then follows from \cite[Corollary 1.18]{leibman} that $\varphi(G)$ lies in a finitely generated subgroup of $N$, and so we may also assume that $N$ is finitely generated. The proposition then follows from \cref{lem:tf.quot} and \cite[Proposition 3.15]{leibman}.
\end{proof}

\section{Hyperbolic groups} \label{yago}

The following argument was communicated by Yago Antol\'in, and shows that generic subgroups of hyperbolic groups are free, with respect to the uniform probability measure on the balls given by a finite generating set. In particular, the degree of nilpotence with respect to such a measure is zero for any non-elementary hyperbolic group.

These techniques and the result are well known to experts, and we include it here for completeness.

\bigskip

As previously, let $F_r$ denote the free group of rank $r$. For a group, $G$, generated by a (finite) set $X$, we let  $\ball_X(n)$ denote the ball of radius $n$, and for an element $g \in G$, we denote by $|g|_X$ the word length of $g$. Let $\mu_n$ be the uniform probability measure on the ball of radius $n$ in $G$ with respect to $X$.

\begin{theorem}\label{thm:generic-free}
	Let $G$ be a non-elementary hyperbolic group with finite generating set $X$. For every $r\in \mathbb{N}$
	$$\lim_{n\to \infty} \dfrac{|\{(g_1,\dots,g_r)\in \ball_X(n)^r\mid \gen{g_1,\dots,g_r}\cong F_r\}|}{|\ball_X(n)|^r}=1, $$
	and the limit converges exponentially fast.
\end{theorem}
We note that the analogous theorem with respect to sequences of measures $(\mu^{*n})_{n=1}^\infty$ corresponding to the steps of the random walk on $G$ was proved in \cite{GMO}.

The following Corollary is immediate:

\begin{corollary}\label{hypnonilp}
Let $G$ be a non-elementary hyperbolic group with finite generating set, $X$, and write $M=(\mu_n)_{n=1}^\infty$ for the sequence of uniform measures on the balls $\ball_X(n)$. Then $\dc^k_M(G)=0$.
\end{corollary}

Throughout, $G$ is a non-elementary hyperbolic group (i.e.\ a hyperbolic group that is not virtually cyclic) and $X$ a finite generating set of $G$.
We assume that $\ga(G,X)$ is $\delta$-hyperbolic. There are many equivalent definitions of Gromov hyperbolicity, (see, for example, \cite[Proposition III.H.1.17]{BH}), for convenience we will use the one that says that geodesic triangles are  $\delta$-thin. In particular, if $x,y,z\in G$, and $\alpha$ is a geodesic with endpoints in $x$ and $y$, $\beta$ a geodesic with endpoints in $x,z$ and $\gamma$ a geodesics with endpoints $y,z$ then we have that for points $v\in \alpha$ and $u\in \beta$ with $\dis(x,u)=\dis(x,v) \leq (y\cdot z)_x \coloneqq\frac{1}{2}(\dis(x,y)+\dis(x,z)-\dis(y,z))$ one has that $\dis(u,v)\leq \delta$. 

Since $G$ has exponential growth, $\lim \sqrt[n]{|\ball_X(n)|}=\lambda >1$. A result of Coornaert \cite{c93} states that there are positive constants $A$, $B$ and $n_0$ such that 
\begin{equation}\label{eq:growth}
A\lambda^n \leq |\ball_X(n)|\leq B\lambda^n
\end{equation}
for all $n\geq n_0$. 

\begin{remark}
	From the submultiplicativity of the function $|\ball_X(n)|$ it follows that $\lim \sqrt[n]{|\ball_X(n)|}$ exists and hence  for every $\varepsilon >0$ there exists $n_\varepsilon$, $A$ and $B$ such that for all $n>n_\varepsilon$, \[ A(\lambda-\varepsilon)^n\leq |\ball_X(n)|\leq B(\lambda + \varepsilon)^n. \]
	One can prove \cref{thm:generic-free} using this weaker fact. However, for simplicity of exposition, we have preferred to use  \eqref{eq:growth}.
\end{remark}

\begin{lemma}[Delzant {\cite[Lemma 1.1.]{Delzant}}] \label{lem:Delzant}
	Let $(x_n)$ be a sequence of points on a $\delta$-hyperbolic geodesic metric space such that
	$\dis(x_{n+2},x_n)\geq \max \left(\dis(x_{n+2},x_{n+1}), \dis(x_{n+1},x_{n})\right)+2\delta+a$. 
	Then  $\dis(x_n,x_m)\geq a |m-n|$.
\end{lemma}

\begin{lemma}\label{lem:sufficient-condition-free}
	There exists a constant $D_0=D_0(\delta)\geq 0$ such that the following holds. 
	
	Let $g_1,g_2,\dots, g_r \in G$  satisfying that  for all 
	$a,b\in \{g_1,\dots, g_r\}^{\pm 1}$ with $a\neq b^{-1}$ the inequality 
	\begin{equation}\label{eq:sufficient-condition-free}
	|ab|_X\geq \max\{|a|_X, |b|_X\}+D_0
	\end{equation} 
	holds. Then $\gen{g_1,\dots,g_n}$ is a free subgroup with basis $\{g_1,\dots,g_n\}$.
\end{lemma}
\begin{proof}
	Take $D_0\geq 2\delta +1$. Let $w$ be any reduced word on $Z=\{g_1,\dots,g_r\}^{\pm 1}$ and denote by $w_i$ the prefix of length $i$ (as a word in $Z$). Then $\dis(w_i,w_{i+2})=\dis(1,w_i^{-1}w_{i+2})=|ab|_X$ for some $a,b \in Z$ with $a\neq b^{-1}$ (since $w$ is reduced).
	Thus, it follows from \cref{lem:Delzant} that $|w|_X\geq \ell_Z(w)$, where $\ell_Z(w)$ denotes the length of $w$ as a word in $Z$.
\end{proof}

We will find bounds on the number of elements in $\ball_X(n)$ not satisfying \eqref{eq:sufficient-condition-free}.
There are two different cases to be considered: $a=b$ and $a\neq b$. 

\begin{lemma}\label{lem:AA}
	There is $D_1=D_1(\delta,D_0)\geq 0$ such that the cardinality of the set
	\[ AA(n)=\{g\in \ball_X(n)\mid |g^2|_X< |g|_X+D_0\} \] is bounded above by $|\ball_X(\frac{n}{2}+D_1)|.$
\end{lemma}
\begin{proof}
	Let $g\in \ball_X(n)$ with $|g^2|_X < |g|_X+D_0$. 
	Then $(1 \cdot g^2)_g > |g|_X/2 -D_0/2$. Let $w$ be a geodesic word over $X$ representing $g$. 
	Suppose that $w=w_\iota w' w_\tau$, where $w_\iota$ and $w_\tau$ are the prefix and suffix of $w$ of length $|g|_X/2 - D_0/2$, respectively. Then, there exists $t\in \ball_X(\delta)$ such that $w_\tau w_\iota =_G t$. Thus $w_\iota^{-1}gw_{\iota}=w't$, and therefore $g$ is conjugated to an element of length at most $D_0+\delta$ by an element of length at most $n/2- D_0/2$. Hence, the cardinality of $AA(n)$ is bounded above by $|\ball_{X}(n/2-D_0/2)| |\ball_{X}(D_0+\delta)|\leq |\ball_X(\frac{n}{2}+D_1)|$ for some $D_1$.
\end{proof}

\begin{remark}
	Note that by  \cref{lem:sufficient-condition-free}, if $g\notin AA(n)$ then $g$ has infinite order. Thus, in particular, the above Lemma implies that the number of finite order elements in the ball of radius $n$ is at most $|\ball_X(\frac{n}{2}+D_1)|$. This appears  in \cite{Dani}.
\end{remark}

\begin{lemma}\label{lem:AB}
	Let $\e\in(3/4,1)$, $n\in \mathbb{N}$ and $g\in G$. Suppose that $|g|_X>\e n$. Then there exists $D_2=D_2(\delta,D_0,\e)\geq 0$ such that the cardinality of the set
	\[ AB(g,n)=\{h\in \ball_X(n)\mid |h|_X>\e n, |gh|_X< n+D_0\} \] is bounded above by $|\ball_X(\frac{3n}{4}+D_2)|.$
\end{lemma}
\begin{proof}
	Let $h\in \ball_X(n)$ with $|h|_X>\e n$ and $|gh|_X < n+D_0$. 
	Then $$(1 \cdot gh)_g > \e n-n/2 -D_0/2 > n/4- D_0/2.$$ 
	Let $u$ and $v$ be  geodesic words over $X$ representing $g$ and $h$ respectively. 
	Suppose that $u=u_1u_2$ and $v=v_1v_2$, where $u_2$ and  $v_1$ have length $n/4-D_0/2$. Note that $|v_2|_X\leq 3n/4+ D_0/2$.
	Then there exists $t\in \ball_X(\delta)$ such that $u_2 v_1 =_G t$. Thus, $u_2h=tv_2$, and so $AB(g,n)$ is contained in $u_2^{-1}\ball_X(\delta)\ball_X(3n/4+D_0/2)$. 
\end{proof}

\begin{proof}[Proof of \cref{thm:generic-free}]
	Fix $\e\in (3/4,1)$. 
	
	Let $$P_0(n)=\dfrac{|\{(g_1,\dots,g_r)\in (\ball_X(n)-\ball_X(\e n))^r \mid g_i\notin AA(n)\}|}{|\ball_X(n)|^r}.$$

	For  $n \gg 0$, we have from \cref{lem:AA} and \eqref{eq:growth}

	\begin{align*}
	P_0(n)&\geq \dfrac{(|\ball_{X}(n)|- B\lambda^{\e n}-B\lambda^{n/2+D_1})^r}{|\ball_X(n)|^r}\\
	&\geq \dfrac{(|\ball_{X}(n)|- 2B\lambda^{\e n})^r}{|\ball_X(n)|^r}\\
	&\geq 1-\dfrac{ \sum_{k=1}^r \binom{r}{k} |\ball_{X}(n)|^{r-k}(2B\lambda^{\e n})^k}{|\ball_X(n)|^r}\\
	&\geq 1- \sum_{k=1}^r \binom{r}{k} \left(\dfrac{2B\lambda^{\e n}}{A \lambda^n}\right)^k\\
	&\geq 1-\dfrac{ C_1}{\lambda^{(1-\e)n}}
	\end{align*}
	where $C_1$ is some constant depending on $A$, $B$ and $r$.
	
	For $j=1,\dots,r$, let $$P_j(n)=\dfrac{|\{(g_1,\dots,g_r)\in (\ball_X(n)-\ball_X(\e n))^r \mid g_j\notin AB(g_i^{\pm 1},n)\text{ for } i\neq j\}|}{|\ball_X(n)|^r}.$$
	For  $n\gg0$, we have from \cref{lem:AB} and \eqref{eq:growth}
	\begin{align*}
	P_j(n)
	&\geq \dfrac{(|\ball_{X}(n)|- |\ball_X(\e n)|)^r -(2r-2)|\ball_X(3n/4+D_2)|(|\ball_{X}(n)|- |\ball_X(\e n)|)^{r-1}}{|\ball_X(n)|^r}\\
	&\geq \dfrac{(|\ball_{X}(n)|- |\ball_X(\e n)|)^r}{|\ball_X(n)|^r} -\dfrac{(2r-2)|\ball_X(3n/4+D_2)|(|\ball_{X}(n)|)^{r-1}}{|\ball_X(n)|^r}\\
	&\geq 1-\dfrac{ C_1}{\lambda^{(1-\e)n}}- \dfrac{(2r-2)B\lambda^{3n/4+D_2}}{A\lambda^n}\\
	&\geq 1-\dfrac{ C_1}{\lambda^{(1-\e)n}} -\dfrac{C_2}{\lambda^{n/4}}
	\end{align*}
	where $C_2$ is some constant depending on $A$,$B$ and $r$.

	Thus, for $i=0,1,\dots,r$ $\lim_{n\to\infty}P_i(n)=1$ converges exponentially fast. By \cref{lem:sufficient-condition-free} we have that for $n\gg 0$
	$$1\geq \dfrac{|\{(g_1,\dots,g_r)\in \ball_X(n)^r\mid \gen{g_1,\dots,g_r}\cong F_r\}|}{|\ball_X(n)|^r}\geq 1 - \sum_{i=0}^r (1-P_i(n))$$
	and taking limits, we see that the probability that an $r$-tuple freely generates a free group converges to $1$ exponentially fast.
\end{proof}

%

\begin{thebibliography}{10}
	
	\bibitem{amv}
	Y.~Antol\'{i}n, A.~Martino, and E.~Ventura, \emph{Degree of commutativity of
		infinite groups}, Proc. Amer. Math. Soc. \textbf{145} (2017), 479--485.
	
	\bibitem{bl}
	V.~Bergelson and A.~Leibman, \emph{A nilpotent {R}oth theorem}, Invent. Math.
	\textbf{147} (2002), no.~2, 429--470.
	
	\bibitem{BH}
	M.~R. Bridson and A.~Haefliger, \emph{Metric spaces of non-positive curvature},
	Grundlehren der Mathematischen Wissenschaften [Fundamental Principles of
	Mathematical Sciences], vol. 319, Springer-Verlag, Berlin, 1999.
	
	\bibitem{c93}
	M.~Coornaert, \emph{Mesures de {P}atterson-{S}ullivan sur le bord d'un espace
		hyperbolique au sens de {G}romov}, Pacific J. Math. \textbf{159} (1993),
	no.~2, 241--270.
	
	\bibitem{Dani}
	P.~Dani, \emph{Genericity of infinite-order elements in hyperbolic groups},
	preprint, available at
	\url{https://www.math.lsu.edu/~pdani/research/hyp.pdf}.
	
	\bibitem{dsp}
	C.~de~Seguins~Pazzis, \emph{The classification of large spaces of matrices with
		bounded rank}, Israel J. Math. \textbf{208} (2015), no.~1, 219--259.
	
	\bibitem{Delzant}
	T.~Delzant, \emph{Sous-groupes \`a deux g\'en\'erateurs des groupes
		hyperboliques}, Group theory from a geometrical viewpoint ({T}rieste, 1990),
	World Sci. Publ., River Edge, NJ, 1991, pp.~177--189.
	
	\bibitem{eberhard.thesis}
	S.~Eberhard, \emph{Some combinatorial problems in group theory}, PhD thesis, University of Oxford, 2016, \url{https://ora.ox.ac.uk/objects/uuid:b92af6aa-df2a-4634-882d-236d8f828857}.
	
	\bibitem{erl}
	A.~Erfanian, R.~Rezaei, and P.~Lescot, \emph{On the relative commutativity
		degree of a subgroup of a finite group}, Comm. Algebra \textbf{35} (2007),
	no.~12, 4183--4197.
	
	\bibitem{gallagher}
	P.~X. Gallagher, \emph{The number of conjugacy classes in a finite group},
	Math. Z. \textbf{118} (1970), 175--179.
	
	\bibitem{GMO}
	R.~Gilman, A.~Miasnikov, and D.~Osin, \emph{Exponentially generic subsets of
		groups}, Illinois J. Math. \textbf{54} (2010), no.~1, 371--388.
	
	\bibitem{gt}
	B.~J. Green and T.~C. Tao, \emph{The quantitative behaviour of polynomial
		orbits on nilmanifolds}, Ann. of Math. (2) \textbf{175} (2012), no.~2,
	465--540.
	
	\bibitem{hall}
	P.~Hall, \emph{The {E}dmonton notes on nilpotent groups}, Queen Mary College
	Mathematics Notes, 1979.
	
	\bibitem{hr}
	K.~H. Hofmann and F.~G. Russo, \emph{The probability that x and y commute in a compact group}, Math. Proc. Camb. Phil. Soc. \textbf{153} (2012), no.~3,
	557--571.
	
	\bibitem{leibman}
	A.~Leibman, \emph{Polynomial mappings of groups}, Israel J. Math. \textbf{129}
	(2002), 29--60.
	
	\bibitem{lp}
	L.~L\'evai and L.~Pyber, \emph{Profinite groups with many commuting pairs or
		involutions}, Arch. Math. (Basel) \textbf{75} (2000), no.~1, 1--7.
		
		\bibitem{ls}
A. Lubotzky and D. Segal. \textit{Subgroup Growth}, Progress in Mathematics \textbf{212}, Birkh\"auser Verlag, Basel (2003).
	
	\bibitem{mpty}
	T.~Meyerovitch, I.~Perl, M.~C.~H. Tointon, and A.~Yadin, \emph{Polynomials and
		harmonic functions on discrete groups}, Trans. Amer. Math. Soc. \textbf{369}
	(2017), 2205--2229.
	
	\bibitem{msc}
	M.~R.~M. Moghaddam, A.~R. Salemkar, and K.~Chiti, \emph{$n$-isoclinism classes
		and $n$-nilpotency degree of finite groups}, Algebra Colloq. \textbf{12}
	(2005), no.~2, 255--261.
	
	\bibitem{neumann}
	P.~M. Neumann, \emph{Two combinatorial problems in group theory}, B. Lond.
	Math. Soc. \textbf{21} (1989), no.~5, 456--458.
	
	\bibitem{shalev}
	A.~Shalev, \emph{Probabilistically nilpotent groups}, Proc. Amer. Math. Soc.
	\textbf{146} (2018), 1529--1536.
	
	\bibitem{comm.prob}
	M.~C.~H. Tointon, \emph{Commuting probabilities of infinite groups}, to appear in J. London Math. Soc., preprint
	available at arXiv:1707.05565 [math.GR], 2017.
	
	\bibitem{verardi}
	L.~Verardi, \emph{Gruppi semiextraspeciali di esponente $p$}, Ann. Mat. Pura
	Appl. \textbf{148} (1987), no.~1, 131--171.
	
\end{thebibliography}

\providecommand{\bysame}{\leavevmode\hbox to3em{\hrulefill}\thinspace}
\providecommand{\MR}{\relax\ifhmode\unskip\space\fi MR }
\providecommand{\MRhref}[2]{%
	\href{http://www.ams.org/mathscinet-getitem?mr=#1}{#2}
}
\providecommand{\href}[2]{#2}

\end{document}